\documentclass[12pt,a4paper,reqno]{amsart}
\usepackage{amssymb,amsmath}
\usepackage{amsfonts,amsbsy,bm}
\usepackage[dvipsnames]{xcolor}
\usepackage{latexsym}
\usepackage{exscale}

\usepackage{amsthm}

\usepackage{mathtools}
\usepackage{MnSymbol}

\DeclarePairedDelimiter\floor{\lfloor}{\rfloor}

\usepackage[colorlinks, bookmarks=true]{hyperref}
\usepackage[utf8]{inputenc}
\usepackage[shortlabels]{enumitem}
\usepackage{color}

\newcommand{\R}{\mathbb R}
\newcommand{\Z}{\mathbb Z}
\newcommand{\N}{\mathbb N}

\newcommand{\F}{\mathbb F}
\newcommand{\B}{\mathcal B}

\newcommand{\cG}{\mathcal G}

\newcommand{\be}{{\mathbf e}}

\newcommand{\GG}{{\mathcal{G}}}

\newcommand {\X} {{\mathbb X}}
\newcommand {\Y} {{\mathbb Y}}

\newcommand {\e} {{\varepsilon}}

\newcommand{\bfe}{{\boldsymbol\e}}

\newcommand {\dl} {{\dot{\ell_{\eta}^1}}}

\def\supp{\mathop{\rm supp}}
\def\sgn{\mathop{\rm sign}}

\numberwithin{equation}{section}
\newtheorem*{theorem*}{Theorem}
\newtheorem{theorem}{Theorem}[section]
\newtheorem{lemma}[theorem]{Lemma}
\newtheorem{defi}[theorem]{Definition}
\newtheorem{corollary}[theorem]{Corollary}
\newtheorem{Remark}[theorem]{Remark}
\newtheorem{remark}[theorem]{Remark}
\newtheorem{proposition}[theorem]{Proposition}
\newtheorem{definition}[theorem]{Definition}
\newtheorem{question}{Question}
\newtheorem{example}[theorem]{Example}
\theoremstyle{definition}

\newcommand{\Ba}[1]{\begin{array}{#1}}
	\newcommand{\Ea}{\end{array}}
\newcommand{\Be}{\begin{equation}}
	\newcommand{\Ee}{\end{equation}}
\newcommand{\Bea}{\begin{eqnarray}}
	\newcommand{\Eea}{\end{eqnarray}}
\newcommand{\Beas}{\begin{eqnarray*}}
	\newcommand{\Eeas}{\end{eqnarray*}}
\newcommand{\Benu}{\begin{enumerate}}
	\newcommand{\Eenu}{\end{enumerate}}
\newcommand{\Bi}{\begin{itemize}}
	\newcommand{\Ei}{\end{itemize}}

\newcommand{\BR}{\begin{Remark} \em}
	\newcommand{\ER}{\end{Remark}}
\newcommand{\BE}{\begin{example} \em}
	\newcommand{\EE}{\end{example}}

\newcounter{reg}
\newcounter{regTO}
\newcommand{\bff}{\mathbf 1}

\newcommand{\n}{\mathbf{n}}
\newcommand{\m}{\mathbf{m}}
\newcommand\G{\mathbf{G}}
\newcommand{\K}{\mathbf{K}}
\newcommand{\one}{\mathbf{1}}
\newcommand{\C}{\mathbf{C}}
\usepackage[shortlabels]{enumitem}

\title[Greedy-like bases for sequences with gaps]{Greedy-like bases for sequences with gaps}

\author[M. Berasategui]{Miguel Berasategui}

\address{Miguel Berasategui
	\\
	IMAS - UBA - CONICET - Pab I, Facultad de Ciencias Exactas y Naturales \\ Universidad de Buenos Aires \\ (1428), Buenos Aires, Argentina}
\email{mberasategui@dm.uba.ar}

\author[P.\ M. Bern\'a]{Pablo M. Bern\'a}
\address{Pablo M. Bern\'a\\
	Departamento de M\'etodos Cuantitativos\\ CUNEF Universidad\\ Madrid, 28040 Spain.}
\email{pablo.berna@cunef.edu}

\begin{document}
	
	\thanks{The first author was supported by ANPCyT  PICT-2018-04104. The second author was supported by the grant PID2019-105599GB-I00 (Agencia Estatal de Investigación, Spain).
	}
	
	\begin{abstract}
		
		%In \cite{O2015}, T. Oikhberg introduced and studied variants of the greedy and weak greedy algorithms for sequences with gaps, with a focus on the $\n$-$t$-quasi-greedy property that is based on them. In this work, we continue the study of these algorithms and bases introducing new ideas for two purpose: prove that for $\n$ with bounded quotient gaps, $\n$-$t$-quasi-greedy bases are quasi-greedy, generalizing the result proved in \cite{BB} to the context of Markushevich bases, and completes the answer to a question from \cite{O2015}. The second purpose is the extension to the context of sequences with gaps other greedy-like properties to study the convergence of the greedy algorithm.
		
		In \cite{O2015}, T. Oikhberg introduced and studied variants of the greedy and weak greedy algorithms for sequences with gaps, with a focus on the $\n$-$t$-quasi-greedy property that is based on them. Building upon this foundation, our current work aims to further investigate these algorithms and bases while introducing new ideas for two primary purposes. Firstly, we aim to prove that for $\n$ with bounded quotient gaps, $\n$-$t$-quasi-greedy bases are quasi-greedy bases. This generalization extends the result previously established in \cite{BB} to the context of Markushevich bases and, also, completes the answer to a question from \cite{O2015}. The second objective is to extend certain approximation properties of the greedy algorithm to the context of sequences with gaps and study if there is a relationship between this new extension and the usual convergence.
	\end{abstract}
	\maketitle
	
	\section{Introduction and background}\label{sectionintroduction}
	
	Let $\X$ be a separable, infinite dimensional Banach space over the field $\mathbb F=\mathbb R$ or $\mathbb C$, with dual space $\X^*$. A \textit{fundamental minimal system}  $\B=(\be_i)_{i\in \N}\subset \X$ is a sequence that satisfies the following: 
	\begin{enumerate}[\rm i)\color{black}]
		\item $\X=\overline{[\be_i\,\colon\, i\in\N]}$;
		\item there is a (unique) sequence $\B^*=(\be_i^*)_{i=1}^{\infty}\subset \X^*$ of biorthogonal functionals, that is, $\be_k^*(\be_i)=\delta_{k,i}$ for all $k,i\in\N$.
	\end{enumerate}
	If $\B$ verifies the above conditions and 
	$$
	\be_i^*(x)=0\qquad \forall i\in\N\Longrightarrow x=0 \qquad\text{(totality)},
	$$
	we say that $\B$ is a \emph{Markushevich basis}. If there is also a positive constant $\C$ such that
	$$\Vert S_m(x)\Vert \leq \C \Vert x\Vert \qquad\forall x\in \X,\,\forall m\in \N,$$ where $S_m$ is the $m$th partial sum $\sum_{i=1}^{m}\be_i^*(x)\be_i$, we say that $\B$ is a \textit{Schauder basis}. Its basis constant $\K_b$ is the minimum $\C$ for which this inequality holds. If $\K_b=1$, the Schauder basis is \emph{monotone}, and it is \emph{bimonotone} if $\K_b=1$ and $\Vert x-S_m(x)\Vert \le  \Vert x\Vert$ for all $x\in\X$ and $m\in \N$. \\
	If there is $\C>0$ such that 
	$$\Vert P_A(x)\Vert \leq \C \Vert x\Vert \qquad\forall x\in \X,\, \forall A\subset\N: |A|<\infty,$$ 
	where $P_A=P_{\B,A}$ is the projection on $A$ with respect to $\B$ (that is $P_A(x)= \sum_{i\in A}\be_i^*(x)\be_i$), we say that $\B$ is \emph{suppression unconditional}. The suppression unconditionality constant $\C_{su}$ is the minimum $\C$ for which the above holds. Equivalently (though not necessarily with the same constant), $\B$ is \emph{unconditional} if 
	$$
	\|\sum_{j\in \N}a_j \be_j^*(x)\be_j\|\le \C \|x\|,\quad \forall x\in \X,\,\forall (a_j)_{j\in\N}\subset \F: |a_j|\le 1, \forall j\in \N,
	$$
	for some $\C>0$.\\
	Hereinafter, by a \emph{basis} for $\X$ we mean a fundamental minimal system $\B$ such that both $\B$ and $\B^*$ are semi-normalized, that is 
	$$
	0<\inf_{i\in \N}\min\{\|\be_i\|,\|\be_i^*\|\}\le \sup_{i\in \N}\max\{\|\be_i\|,\|\be_i^*\|\}<\infty. 
	$$
	We will use $\B$ to denote a basis, and we define positive constants $\alpha_1, \alpha_2, \alpha_3$ as follows: 
	$$
	\alpha_1:=\sup_{i\in\N}\|\be_i\|, \qquad\alpha_2:=\sup_{i\in\N}\|\be_i^*\|,\qquad\text{and}\qquad \alpha_3:=\sup_{i\in \N}\|\be_i\|\|\be_i^*\|.
	$$
	
	In 1999, S. V. Konyagin and V. N. Temlyakov introduced the Thresholding Greedy Algorithm (TGA), which is one of the most important algorithms in he field of non-linear approximation, and has been studied by researchers such as F. Albiac, J. L. Ansorena, S. J. Dilworth, G. Garrigós, E. Hernández, N. J. Kalton, D. Kutzarova, V. N. Temlyakov and P. Wojtaszczyk, among others.  The algorithm essentially chooses for each $x\in \X$ the largest coefficients in modulus with respect to a basis. 
	Here, we consider a relaxed version of this algorithm introduced by V. N. Temlyakov in \cite{Tem2}.  Fix $t\in (0,1]$.
	We say that a set $A(x,t):=A$ is a $t$-\textbf{greedy set} for $x\in\X$ if
	$$\min_{i\in A}\vert\be_i^*(x)\vert\geq t \max_{i\not\in A}\vert\be_i^*(x)\vert.$$
	
	\noindent A $t$-\textbf{greedy sum} of order $m$ (or an $m$-term $t$-greedy sum) is the projection
	$$\mathbf{G}_m^t(x)=\sum_{i\in A}\be_i^*(x)\be_i,$$
	where $A$ is a $t$-greedy set of cardinality $m$. The collection $(\mathbf{G}^t_m)_{m=1}^\infty$ is called the \textbf{Weak Thresholding Greedy Algorithm} (WTGA) (see \cite{Tem,Tem2}), and we denote by $\mathcal G_m^t$ the collection of $t$-greedy sums $\G^t_m$ with $m\in\mathbb N$. Also, we denote by $\GG(x,m,t)$ the set of all $t$-greedy sets for $x$ of cardinality $m$. \\
	If $t=1$, we talk about greedy sets and greedy sums $\G_m$ (see \cite{KT}).\\

	Different types of convergence of these algorithms have been studied in several papers, for instance \cite{DKK2003, DKKT, KT}. For $t=1$, a central concept in these studies is the notion of quasi-greediness (\cite{KT}). 
	
	\begin{defi}
		We say that $\mathcal B$ is quasi-greedy if there exists a positive constant $\mathbf C$ such that
		$$\Vert \mathbf{G}_m(x)\Vert \leq \mathbf C\Vert x\Vert,\; \forall x\in\mathbb X, \forall m\in\mathbb N.$$
	\end{defi}

	The relation between quasi-greediness and the convergence of the algorithm was given by P. Wojtaszczyk in \cite{Wo}: a basis is quasi-greedy if and only 
	$$\lim_n \G_n(x)=x,\; \forall x\in\X.$$\\
	In 2015, T. Oikhberg introduced and studied a variant of the WTGA where only the $t$-greedy sums with order in a given strictly increasing sequence of positive integers $\mathbf{n}=(n_k)_{k=1}^\infty$ are considered  \cite{O2015}.  Oikhberg's central definition is as follows: given $\mathbf{n}=(n_k)_{k=1}^\infty\subset \N$ a strictly increasing sequence $n_1<n_2<...$, a basis $\mathcal B$ is $\n$-$t$-quasi-greedy if
	\begin{eqnarray}\label{qO}
		\lim_k \G_{n_k}^t(x)=x,
	\end{eqnarray}
	for any $x\in\X$ and any choice of $t$-greedy sums $\G_{n_k}^t(x)$.\\
	Of course, if the basis is quasi-greedy, it is $\n$-quasi-greedy for any sequence $\n$; moreover, it is also $\n$-$t$-quasi greedy for all $0<t\le 1$ (see \cite[Theorem 2.1]{O2015}, \cite[Lemmas 2.1, 2.3]{KT2002}, \cite[Proposition 4.5]{DKSWo2012}, \cite[Lemma 2.1, Lemma 6.3]{DKO2015}, and Lemma~\ref{lemmaQG}). The reciprocal is false as \cite[Proposition 3.1]{O2015} shows. \\
	
	The study of $\n$-$t$-quasi-greedy bases can be of interest for the following reasons: first, as pointed out in \cite[Questions 2,3]{O2015}, there are classical spaces that do not have a quasi-greedy basis -  e.g., $C[0,1]$, see \cite{DKKT} - or a uniformly bounded quasi-greedy basis - e.g., $L_1[0,1]$ -, see \cite{DST2012} -  so one may study whether it is possible to obtain bases (or uniformly bounded bases) where the greedy algorithm at least converges through a subsequence. 
	Second, and in a more general setting, the study of $\n$-$t$-quasi greedy bases sheds light on general properties of the convergence of the greedy algorithm, advancing the theory on the TGA. For example, quasi-greedy bases are the weakest bases for which the greedy sums converge for all $x\in \X$. If we can find the sequences for which $\n$-quasi-greedy bases are quasi-greedy, we obtain formally weaker conditions that guarantee said convergence, and a  method that may simplify the construction of quasi-greedy bases with properties of interest, as well as the determination of whether a basis is quasi-greedy.  A partial result in this direction was proven in \cite[Proposition 4.1]{O2015}, for sequences that are \emph{crude bases} (see \cite{O2015}). In this work, we make further progress on this front, completing the characterization of the sequences for which the aforementioned implication holds - which is based on the following classification. 
	\begin{definition}\label{definitiondifferentgaps}
		Let $\n=(n_k)_{k\in \N}$ be a strictly increasing sequence of natural numbers. The \emph{quotient gaps} of the sequence are the quotients $\frac{n_{k+1}}{n_k}$, when $n_{k+1}>n_{k}+1$, and we say that $\n$ has arbitrarily large quotient gaps if
		$$\limsup_{k\rightarrow+\infty}\frac{n_{k+1}}{n_k}=+\infty.$$
		
		\noindent Alternatively, for $l\in\mathbb N_{>1}$, we say that $\n$ has $l$-bounded quotient gaps if it has gaps - i.e., $\n\not=\N$ - and 
		$$\frac{n_{k+1}}{n_k}\leq l,$$
		for all $k\in\mathbb N$, and we say that it has bounded quotient gaps if it has $l$-bounded quotient gaps for some natural number $l\ge 2$.
	\end{definition}
	In \cite[Proposition 3.1]{O2015}, it was proven that if $\n$ has arbitrarily large quotient gaps, there are $\n$-$t$-quasi greedy bases that are not quasi-greedy; in fact, the construction gives Schauder bases with these properties. In the other direction, the author asked whether  $\n$-$t$ quasi-greedy bases that are not quasi-greedy are possible possible for some sequences that grow exponentially or even more slowly (\cite[Question 6.4]{O2015}). A partial result was obtained in \cite[Theorem 5.2]{BB}, where it was shown that the answer is negative for Schauder bases. However, the general question remained open; the main obstacle was that some the techniques used in \cite{O2015} and \cite{BB} to prove the previously mentioned partial results seem to rely crucially on further conditions on either the sequence or the basis, and thus appear not suitable for the general case. In this work, we combine recent results from \cite{AAB2023} and \cite{BB2} and use a threshold property from \cite{DKKT} to complete the characterization: $\n$-$t$-quasi-greediness entails quasi-greediness if and only if $\n$ has bounded quotient gaps. \\
	In addition to the study of $\n$-$t$-quasi-greedy bases, we continue the general study of the greedy algorithm for sequences with gaps, extending to our context some of the notions of greedy-type bases, with an eye on either finding formally weaker characterizations of well-known greedy-like bases or - alternatively - finding cases where it is possible to have only a weaker form of convergence of a greedy-like algorithm through a  subsequence. \\
	
	This paper is organized as follows:  in Section \ref{sectionQGgaps}, we study $\n$-$t$-quasi-greedy bases  and prove our main result, Theorem~\ref{theoremnQGboundedgaps->QG!!!}. In Section \ref{sectionbidemocracy}, we define $\n$-bidemocracy and study duality of $\n$-$t$-quasi-greedy bases. In Sections~\ref{sectionuncond} and~\ref{sectionsemi}, we study the extensions of unconditionality and semi-greediness to our context, respectively. In Section \ref{sectionpartially}, we introduce and study $\n$-partially greedy and $\n$-strong partially greedy bases. In Section~\ref{sectionexamples}, we consider a family of examples that is used throughtout the paper. Finally, in Section~\ref{sectionquestions}, we pose some questions for future research.

	We will also use the following notation throughout the paper - in addition to that already introduced: 
	for $A$ and $B$ subsets of $\mathbb N$, we write $A<B$ to mean that  $\max A<\min B$. If $m\in \mathbb N$, we write $m < A$ and $A < m$ for $\{ m\} < A$ and  
	$A <\{ m\}$ respectively, and we use the symbols ``$>$'', ``$\ge$'' and ``$\le$'' similarly.\\
	Also, $A\cupdot B$ means the union of $A$ and $B$ with $A\cap B=\emptyset$, and $\N_{>k}$ means the set $\N\setminus \lbrace 1,\dots,k\rbrace$. \\
	
	For $A\subset\N$ finite, $\Psi_A$ denotes the set of all collections of sequences $\bfe = (\e_i)_{i\in A}\subset \mathbb F$ such that $\vert \varepsilon_n\vert=1$ and
	$$
	\bff_{\bfe A}[\mathcal B,\X]:=\bff_{\bfe A}=\sum_{i\in A}\e_i \be_i.
	$$
	If $\bfe\equiv 1$, we just write $\bff_A$. Also, every time we have index sets $A\subset B$ and $\bfe\in \Psi_B$, we write $\bff_{\bfe A}$ considering the  natural restriction of $\bfe$ to $A$, with the convention that $\bff_{\bfe A}=0$ if $A=\emptyset$. \\
	As usual, by $\supp{(x)}=\supp_{\B}(x)$ we denote the support of $x\in \X$, that is, the set $\{i\in \mathbb{N}: \be_i^*(x)\not=0\}$, and $P_A$ with $A$ a finite set denotes the projection operator, that is,
	$$P_A(x)=\sum_{i\in A}\be_i^*(x)\be_i,$$
	as before with the convention that the sum is zero if $A=\emptyset$.  Also, for $x\in \X$, $\varepsilon(x)$ denotes de sequence of signs $(\sgn(\be_i^*(x))_{i\in \N}$, where $\sgn(0):=1$.
	
	Although T. Oikhberg defined $\n$-$t$-quasi-greedy bases using the condition \eqref{qO}, for our purposes we use the following  definition - which is equivalent by \cite[Theorem 2.1]{O2015}.
	
	\begin{defi}\label{definitionqg}
		Let $\n=(n_k)_{k=1}^\infty$ be a strictly increasing sequence of natural numbers. We say that $\mathcal B$  is $\mathbf{n}$-$t$-\textbf{quasi-greedy} if there exists a positive constant $\C$ such that
		\begin{eqnarray}\label{q}
			\Vert \mathbf{G}^t_n(x)\Vert\leq \C\Vert x\Vert,\; \forall x\in\mathbb X, \forall \G_n^t(x)\in\mathcal G_n^t(x), \forall n\in\mathbf{n},
		\end{eqnarray}
		where $\mathcal G_n^t(x)$ denotes the set of all $n$-term $t$-greedy sums of $x$. \\
		Alternatively, we say that $\mathcal B$ is $\n$-$t$-\textbf{suppression-quasi-greedy} if there exists a positive constant $\C$ such that 
		\begin{eqnarray}\label{sq}
			\Vert x-\mathbf{G}^t_n(x)\Vert \leq \C\Vert x\Vert,\; \forall x\in\mathbb{X}, \forall \G_n^t(x)\in\mathcal G_n^t(x), \forall n\in\n.
		\end{eqnarray}
		We denote by $\C_{q,t}$ and $\C_{sq,t}$ the smallest constants verifying \eqref{q} and \eqref{sq}, respectively, and we say that $\mathcal B$ is $\C_{q,t}$-$\n$-$t$-quasi-greedy and $\C_{sq,t}$-$\n$-$t$-suppression-quasi-greedy.
	\end{defi}
	
	While a sequence $\n$ has gaps if and only if $\n\not=\N$, for convenience we will allow $\n=\N$ in our definitions and statements unless otherwise specified. 
	
	\begin{remark}
		For $\n=\mathbb N$ and $t=1$, we recover the classical definitions of quasi-greediness (see \cite{KT}).\end{remark}
	
	We will use the following notation for $\n$-$t$-quasi-greedy bases:
	\begin{itemize}
		\item If $\n=\mathbb N$, we say that $\mathcal B$ is $\mathbf C_{q,t}$-$t$-quasi-greedy (resp. $\mathbf C_{sq,t}$-$t$-suppression-quasi-greedy).
		\item If $t=1$, we say that $\mathcal B$ is $\mathbf C_{q}$-$\n$-quasi-greedy (resp. $\mathbf C_{sq}$-$\n$-suppression-quasi-greedy).
		\item If $t=1$ and $\n=\mathbb N$, we say that $\mathcal B$ is $\mathbf C_{q}$-quasi-greedy (resp. $\mathbf C_{sq}$-suppression-quasi-greedy).
		\item If $\n$ and $\m$ are two sequences and $\B$ is both $\n$-$t$-QG and $\m$-$t$-QG, we write $\C_{sq,t,\n}$ (or $\C_{q,t,\n}$) and $\C_{sq,t,\m}$ (resp. $\C_{q,t,\m}$) to distinguish them, also with the convention that we do not write $t$ when $t=1$. 
	\end{itemize}

	Finally, we let 
	\begin{align*}
		\kappa:=\begin{cases}
			1 & \text {if } \F=\R,\\
			2 & \text{if }\F=\mathbb{C}.
		\end{cases}
	\end{align*}

	\section{$\n$-Quasi-greedy bases.}\label{sectionQGgaps}
	
	In this section, we continue the study of $\n$-$t$-quasi-greedy bases and prove our main result, which shows that $\n$-quasi-greedy bases are quasi-greedy if $\n$ has bounded quotient gaps. Additionally, we obtain a result for $\n$-$t$-quasi-greedy bases when $\n$ has arbitrarily large quotient gaps, which gives a partial answer to \cite[Question 6.4]{O2015}. A key property in our study is thresholding boundedness, first introduced in \cite{DKK2003}.

	\begin{definition}\label{definitionthresholdingbounded}
		Let $\B$ be a basis for a Banach space $\X$ and define 
		$$
		\mathcal{Q}:=\{x\in \X: \sup_{i\in \N}|\be_i^*(x)|\le 1\}
		$$
		and, for each $0<t\le  1$ and each $x\in Q$, 
		$$
		A(x,a):=\{i\in \N: |\be_i^*(x)|\ge a\}. 
		$$
		We say that $\B$ is \emph{thresholding bounded} if, for each $0<t\le 1$, there is $\C_t>0$ such that 
		$$
		\|P_{A(x,t)}(x)\|\le \C_t\|x\|\qquad\forall x\in Q. 
		$$
		The function $\theta$ is defined on $(0,1]$ so that for each $0<t\le 1$, $\theta(t)$ is the minimum $\C_t$ for which the above inequality holds. We also define $\theta_c(t)$ as the minimum $\K_t$ such that 
		$$
		\|x-P_{A(x,t)}(x)\|\le \K_t\|x\|\qquad\forall x\in Q. 
		$$
	\end{definition} 
	Note that - by scaling - $\B$ is quasi-greedy if and only if $\theta$ (equivalently, $\theta_c$) is bounded (see for example \cite{DKK2003}).

	\begin{remark}\rm 
		It is known that $\theta$ is non-increasing \cite[Proposition 4.1]{DKK2003}. We note that the same holds for $\theta_c$: Indeed, given $x\in \mathcal{Q}$ and $0<t_1<t_2\le 1$, we have
		\begin{align*}
			\left\Vert x-P_{A(x,t_2)}(x)\right\Vert=&\frac{t_2}{t_1}\left\Vert \frac{t_1x}{t_2} -P_{A\left(\frac{t_1x}{t_2}  ,t_1\right)}\left(\frac{t_1x}{t_2} \right)\right\Vert \le \frac{t_2}{t_1}\theta_c\left(t_1\right)\left\Vert \frac{t_1x}{t_2}\right\Vert =\theta_c\left(t_1\right)\|x\|.
		\end{align*}
	\end{remark}

	For the proof of Theorem~\ref{theoremnQGboundedgaps->QG!!!}, it is convenient to define two more variants of $\theta$ and $\theta_c$ as follows: Let 
	\begin{align*}
		\mathcal{Q}_{0}:=&\left\lbrace x\in \mathcal{Q}: |\be_j^*(x)|\not=|\be_k^*(x)|\forall k, j\in \supp(x): k\not=j \right\rbrace. 
	\end{align*}
	Now given a thresholding bounded basis $\B$, let $\vartheta$ and $\vartheta_c$ be the minima defined as $\theta$ and $\theta_c$ but in terms of elements of $\mathcal{Q}_{0}$ instead of $\mathcal{Q}$. Clearly $\vartheta\le \theta$ and $\vartheta_c\le \theta_c$, and both $\vartheta$ and $\vartheta_c$ are non-increasing for the same reasons that hold for $\theta$ and $\theta_c$ respectively. Also, by density, scaling and a standard small perturbation argument, $\B$ is quasi-greedy if and only if $\vartheta$ (equivalently, $\vartheta_c$) is bounded. 
	\begin{remark}\rm \label{remarkwhyQ0}
		Note that if $x\in \mathcal{Q}_0$, $A\in \cG\left(x,m,1\right)$ for some $m\in \N$, and there is $j\not\in A$ such that $\be_j^*(x)\not=0$, then there is $0<t\le 1$ such that $A=A(x,t)$. This property - which does not hold for elements of $\mathcal{Q}$ in general - will allow us to simplify the proof of Theorem~\ref{theoremnQGboundedgaps->QG!!!}. 
	\end{remark}
	
	\begin{remark} \rm \label{remarkQGtheta}
		Note also that for all $0<t\le 1$, 
		\begin{align*}
			&\theta(t)\le 1+\theta_c(t);&&\theta_c(t)\le 1+\theta(t);\\
			&\vartheta(t)\le\ 1+\vartheta_c(t);&&\vartheta_c(t)\le 1+\vartheta(t);
		\end{align*}
	\end{remark}
	It was proven in \cite[Proposition 4.5]{DKK2003} that thresholding boundedness is equivalent to near unconditionality, a property introduced by Elton in 1978 \cite{Elton1978}. 
	
	\begin{definition}\label{definitionnearlyuncond}
		A basis $\B$ be a basis for a Banach space $\X$  \emph{nearly unconditional} if, for each $0<t\le 1$, there is $\C_t>0$ such that 
		$$
		\|P_B(x)\|\le \C_t\|x\|\qquad\forall x\in Q\;\forall B\subset A(x,t). 
		$$
		The near unconditionality function $\phi$ is defined on $(0,1]$ so that for each $0<t\le 1$, $\phi(t)$ is the minimum $\C_t$ for which the above inequality holds.  
	\end{definition}
	
	Recently, it has be shown that near unconditionality and thresholding boundedness are also equivalent to near truncation quasi-greediness \cite[Theorem 3.4]{AABBL2021b} and to quasi-greediness for largest coefficients (\cite[Theorem 2.6]{AAB2023}. We will use the latter equivalence in the proof of Theorem~\ref{theoremnQGboundedgaps->QG!!!}.
	
	\begin{definition}\cite[Definition 4.6 ]{AABW}\label{definitionnQGlc} We say that $\mathcal B$ is quasi-greedy for largest coefficients if there exists a positive constant $\C$ such that
		\begin{eqnarray}\label{ql}
			\|\bff_{\bfe A}\|\le \C\|\bff_{\bfe A}+x\|
		\end{eqnarray}
		for every $A\subset \N$ finite, $\bfe \in \Psi_A$, and all $x\in \mathcal{Q}$ such that $\supp{(x)}\cap A=\emptyset$. The smallest constant verifying \eqref{ql} is denoted by $\C_{ql}$ and we say that $\B$ is $\C_{ql}$-quasi-greedy for largest coefficients. 
	\end{definition}
	
	Combining the aforementioned equivalence with \cite[Proposition 6]{BB2} we immediately obtain the following result. 
	\begin{corollary}\label{corollary: nQGbounded->NU} Let  $\B$ be a basis for a Banach space $\X$, and $\n$ a sequence with bounded quotient gaps. If $\B$ is $\n$-quasi-greedy, it is thresholding bounded. 
	\end{corollary}
	To further simplify the proof of Theorem~\ref{theoremnQGboundedgaps->QG!!!}, we give first some auxiliary results that will allow us to replace a sequence $\n$ with one that is more suitable to that end. Given a sequence $\n$ and $j\in \N$, following \cite{O2015} by $j*\n$ we denote the only strictly increasing sequence $\m=(m_k)_{k\in \N}$ with the property that for each $m\in \N$, there is $k\in \N$ such that $m=m_k$ if and only if $m$ can be written as a sum 
	$$m=n_{k_1}+\dots+n_{k_{j_1}},$$ where $1\le j_1\le j$ and $k_1,\dots, k_{j_1}$ are any  (possibly repeated) natural numbers.
	\begin{lemma}\label{lemmasimplify}Let $\B$ be a $\C_{sq,t}$-$\n$-QG basis for a Banach space $\X$. Suppose $x\in \X$ and $A$ is a $t$-greedy set for $x$ with $|A|\in m*\n$ . Then 
		\begin{align*}
			\|x-P_A(x)\|\le& \C_{sq,t}^m\|x\|. 
		\end{align*}
	\end{lemma}
	\begin{proof}
		This follows at once from the proof of \cite[Proposition 4.1]{O2015}
	\end{proof}

	\begin{lemma}\label{lemma2bg}Let $l\in \N_{\ge 2}$ and let $\n$ be a sequence with $l$-bounded quotient gaps. Then either $l*\n$ has $2$-bounded quotient gaps or $l*\n=\N$. 
	\end{lemma}
	\begin{proof}
		Suppose $l*\n\not=\N$, pick $1\le l_0\le l$ and choose $(k_j)_{1\le j\le l_0}$  a finite sequence of natural numbers with the property that either $l_0=1$ or $k_{j}\le k_{j+1}$ for all $1\le j\le l_0-1$. Let $m=\sum_{j=1}^{l_0}n_{k_j}$. We have to find $m'\in l*\n$ such that $m<m'\le 2m$. 
		\begin{itemize}
			\item If $l_0=1$, then trivially $2m=2 n_{k_1}\ge m':= n_{k_1}+n_{k_1}>n_{k_1}=m$. 
			\item If $1<l_0<l$, then 
			\begin{align*}
				2m=&2\sum_{j=1}^{l_0}n_{k_j}> m':=n_{k_1}+\sum_{j=1}^{l_0}n_{k_j}>\sum_{j=1}^{l_0}n_{k_j}=m.
			\end{align*}
			\item If $l_0=l>1$, and $k_1<k_{l}$, then 
			\begin{align*}
				2m=&2\sum_{j=1}^{l}n_{k_j}=2 n_{k_1}+2\sum_{j=2}^{l}n_{k_j}>m':= n_{k_{l}}+\sum_{j=2}^{l}n_{k_j}>\sum_{j=1}^{l}n_{k_j}=m. 
			\end{align*}
			\item If $l_0=l>1$ and $k_1=k_{l}$, then $k_j=k_l$ for all $1\le j\le l$. Hence, 
			\begin{align*}
				2m=&2\sum_{j=1}^{l}n_{k_j}=2l n_{k_l}> m':=n_{k_{l+1}}+(l-1)n_{k_l}>\sum_{j=1}^{l}n_{k_j}=m. 
			\end{align*}
			As we have considered all possible cases, the proof is complete. 
		\end{itemize}
	\end{proof}
	\begin{corollary}\label{corollarysimpler}Let $\n$ be a sequence with $l$-bounded quotient gaps, and $\B$ an $\n$-quasi-greedy basis for a Banach space $\X$. There is a sequence $\m=(m_k)_{k\in \N}$ with  $m_1=1$ such that $\B$ is $\m$-quasi-greedy with $\C_{sq, \m}\le \max\{ 1+\alpha_1\alpha_2 (n_1-1),\C_{sq,\n}^l\}$, and either $\m$ has $2$-bounded quotient gaps or $\m=\N$. 
	\end{corollary}
	\begin{proof}
		By Lemma~\ref{lemmasimplify}, $\B$ is $l*\n$-QG with $\C_{sq, l*\n}\le \C_{sq,\n}^l$, and by Lemma~\ref{lemma2bg}, either $l*\n$ has $2$-bounded quotient gaps or $l*\n=\N$. If $n_1=1$, there is nothing else to prove. If $n_1>1$,  let $\m$ be the (strictly increasing) sequence obtained from $l*\n$ by adding all natural numbers smaller than $n_1$. Then it is immediate that if $\m\not=\N$, then $\m$ has $2$-bounded quotient gaps. Given that
		\begin{align*}
			\|x-P_A(x)\|\le (1+(n_1-1)\alpha_1\alpha_2)\|x\|
		\end{align*}
		for every $A\subset \N$ with $|A|<n_1$, we conclude that $\B$ is $\m$-QG with constant as in the statement. 
	\end{proof}

	Now we prove our main result. 
	\begin{theorem}\label{theoremnQGboundedgaps->QG!!!}
		Let $\n$ be a sequence with bounded quotient gaps. If $\B$ is an $\n$-quasi-greedy Markushevich basis for a Banach space $\X$, it is quasi-greedy. 
	\end{theorem}
	\begin{proof}
		By Corollary~\ref{corollarysimpler}, we may assume that $\n$ has $2$-bounded quotient gaps and $n_1=1$.\\
		Suppose, to obtain a contradiction, that $\B$ is not quasi-greedy, so 
		\begin{align*}
			\vartheta_c\left(t\right)&\xrightarrow[t\to 0]{}\infty.
		\end{align*}
		By Corollary~\ref{corollary: nQGbounded->NU}, $\B$ is thresholding bounded.  We have the following:

		\noindent\emph{\textbf{Claim: }}There is $d_1>0$ such that 
		\begin{align}
			\vartheta_{c}(t)\le d_1\left(1+\left(\log\left(\frac{1}{t}\right)\right)^{d_1}\right) \quad\forall 0<t\le 1. \label{paso1}
		\end{align}
		
		To prove the claim first choose $0<t_0<e^{-1}$ so that $\vartheta_c\left(t_0\right)\ge \C_{sq}^{2}$. Fix $0<t\le t_0$, choose $x\in \mathcal{Q}_0$, and set $A:=A\left(x,t^2\right)$. To find an upper bound for $\|x-P_A(x)\|$, we may assume $0\not=P_A(x)\not=x$. 
		Now we proceed by case analysis: \\
		\begin{enumerate}[\rm $\text{Case}$ 1.]
			\item \label{Case1} If $|A|\in \n$,  then by our choice of $t_0$ and the fact that $\vartheta_c$ is non-increasing, we have $\|x-P_A(x)\|\le \C_{sq}\|x\|\le \vartheta_c(t)\|x\| $.  
			\item \label{Case2} If there are $k_1, k_2\in \N$ such that $|A|=n_{k_1}+n_{k_{2}}$, then by the above considerations and  Lemma~\ref{lemmasimplify}, 
			\begin{align*}
				\|x-P_A(x)\|\le& \C_{sq}^{2}\|x\|\le \vartheta_c\left(t\right)\|x\|. 
			\end{align*}
			\item \label{Case3}If neither of the above cases holds, define
			\begin{align*}
				&k_0:=\max_{k\in \N}\{n_k<|A|\}, 
			\end{align*}
			choose $B\subset A$ so that $B\in \cG\left(x, n_{k_0},1\right)$, and let $m_1:=|A|-|B|$. Note that $1\le m_1<n_{k_0}$. Choose $D_1\subset B$ so that $D_1\in \cG\left(x, m_1, 1\right)$ and let   $D_2:=A\setminus B$. Note that $D_{2}\in \cG\left(x-P_{B}(x),m_1,1\right)$. For $j=1,2$, set
			\begin{align*}
				&b_j:=\max_{k\in D_j}\left\vert \be_k^*\left(x\right)\right\vert;&& a_j:=\min_{k\in D_j}\left\vert \be_k^*\left(x\right)\right\vert. 
			\end{align*}
			It follows by construction that
			\begin{align*}
				&t^2\le a_2\le b_2=\max_{k\in D_{2}=A\setminus B}\left\vert \be_k^*(x)\right\vert \le \min_{k\in B}\left\vert \be_k^*(x)\right\vert\le a_1\\
				&\le b_{1}=\max_{k\in A}\left\vert \be_k^*(x)\right\vert=\max_{k\in B}\left\vert \be_k^*(x)\right\vert=\|x\|_{\ell_{\infty}}.
			\end{align*}
			Thus, 
			\begin{align*}
				t^2\le &\min_{k\in A}\left\vert \be_k^*(x)\right\vert \le \frac{\min_{k\in A}\left\vert \be_k^*(x)\right\vert }{\max_{k\in A}\left\vert \be_k^*(x)\right\vert }=\frac{\min_{k\in A}\left\vert \be_k^*(x)\right\vert }{\min_{k\in B}\left\vert \be_k^*(x)\right\vert }\frac{\min_{k\in B}\left\vert \be_k^*(x)\right\vert }{\max_{k\in B}\left\vert \be_k^*(x)\right\vert }\le \frac{a_2}{b_2}\frac{a_1}{b_1}
			\end{align*}
			Hence, there is  $1\le j_0\le 2$ such that
			\begin{align*}
				\frac{a_{j_0}}{b_{j_0}}\ge t. 
			\end{align*}
			\begin{enumerate}[\rm $\text{Case 3}$ a.]
				\item If $j_0=1$, let $x_1:=x-P_{D_1}(x)$. We have 
				\begin{align*}
					&b_1^{-1}x\in \mathcal{Q}_0;&&D_{1}\subset A\left(b_1^{-1} x,t\right).
				\end{align*}
				By Remark~\ref{remarkwhyQ0}, $D_1=A\left(b_1^{-1}x,t_1\right)$ for some $t\le t_1\le 1$, so 
				\begin{align*}
					\left\Vert x_1\right\Vert=&b_1\left\Vert \frac{x}{b_1}-P_{D_1}\left(\frac{x}{b_1}\right)\right\Vert   \le b_1 \vartheta_c\left(t_1 \right)\left\Vert \frac{x}{b_1}\right\Vert \le \vartheta_c\left(t\right)\|x\|. 
				\end{align*}
				Now $A\setminus D_1\in \cG\left(x_1,n_{k_0},1\right)$, so 
				\begin{align}
					&\|x-P_A(x)\|=\|x_1-P_{A\setminus D_1}(x_1)\|\le \C_{sq}\|x_1\|\le \C_{sq}\vartheta_c\left(t\right)\|x\|. \label{3b}
				\end{align}
				\item If $j_0=2$, let $x_1:=x-P_B(x)$. We have
				\begin{align*}
					&\|x_1\|\le \C_{sq}\|x\|;&&b_{2}^{-1}x_1\in \mathcal{Q}_0;&&&D_{2}\subset A\left(b_{2}^{-1}x_1,t\right).
				\end{align*}
				Given that $D_{2}\in \cG\left(b_{2}^{-1}x_1, m_1,1\right)$, as before there is $t\le t_1\le 1$ such that $D_{2}= A\left(b_{2}^{-1}x_1,t_1\right)$, so
				\begin{align*}
					\left \Vert x_1-P_{D_{2}}(x_1)\right\Vert=& b_{2} \left\Vert \frac{x_1}{b_{2}}-P_{D_{2}}\left(\frac{x_1}{b_{2}}\right)\right\Vert\le \vartheta_c\left(t_1\right)\|x_1\|\le  \vartheta_c\left(t\right)\|x_1\|.
				\end{align*}
				Given that $x_1-P_{D_{2}}(x_1)=x-P_A(x)$, a combination of the above estimates gives
				\begin{align}
					\|x-P_A(x)\|\le \C_{sq}\vartheta_c\left(t\right)\|x\|. \label{3c}
				\end{align}
				From the estimates for \ref{Case1} and \ref{Case2}, together with \eqref{3b} and \eqref{3c}, taking supremum  and considering that $0<t\le t_0$ is arbitrary we conclude that 
				\begin{align}
					\vartheta_c\left(t^{2}\right)\le& \C_{sq}\vartheta_c(t)\quad\forall 0<t\le t_0. \label{forclaim3}
				\end{align} 
			\end{enumerate}
		\end{enumerate}

		To complete the proof of the claim, we use a variant of part of the argument of \cite[Proposition 3.4]{AAB2023}. First note that for each $0<t\le t_0$, 
		$$\vartheta_c\left(t^{4}\right)=\vartheta_c\left(\left(
		t^{2}\right)^2\right)\le \C_{sq} \vartheta_c\left(t^2\right)\le \C_{sq}^{2} \vartheta_c\left(t\right).$$
		Inductively, it follows that
		\begin{align}
			\vartheta_c\left(t^{2^n}\right)\le& \C_{sq}^{n}\vartheta_c(t)\quad\forall 0<t\le t_0\forall n\in \N. \label{forlog}
		\end{align}
		Since $\B$ is not quasi-greedy $\C_{sq}>1$ by \cite[Theorem 3.8]{BB}, so there is $d>0$ such that
		\begin{align}
			\C_{sq}=&2^d. \label{eld}
		\end{align}
		Now pick $0<a<t_0^2$, and choose $t_0^2\le t\le t_0$ and $n\in \N$ so that $a=t^{2^n}$. From \eqref{forlog} and \eqref{eld}, considering that $0<t\le t_0<e^{-1}$ and $\vartheta$ is non-increasing, it follows that
		\begin{align*}
			\vartheta_c\left(a\right)=&\vartheta_c\left(t^{2^n}\right)\le \C_{sq}^{n}\vartheta_c(t)=2^{dn}\vartheta_c\left(t\right)=\left(\frac{\log\left(\frac{1}{t^{2^n}}\right)}{\log\left(\frac{1}{t}\right)}\right)^d\vartheta_c(t)\le \vartheta_c\left(t_0^2\right) \left(1+\log^d\left(\frac{1}{a}\right)\right)\\
			\le & d_1\left(1+\log^{d_1}\left(\frac{1}{a}\right)\right),
		\end{align*}
		where 
		\begin{align*}
			d_1:=&\max\left\lbrace d,  \vartheta_c\left(t_0^2\right)\right\rbrace. 
		\end{align*}
		Since $\vartheta$ is non-increasing, $\vartheta_c(t)\le d_1$ for all $t_0^2\le t\le 1$, and our claim is proven. \\
		Now fix $0<t<1$, $x\in \mathcal{Q}_0$, and let $A:=A(x,t)$. As before, we may assume $0\not=P_A(x)\not=x$. We consider the following cases: \\
		\begin{enumerate}[\rm $\text{Case}$ (i) ]
			\item \label{casec1}  If $|A|\in \n$ or $|A|\in 2*\n$, then by  Lemma~\ref{lemmasimplify} 
			\begin{align*}
				\|x-P_A(x)\|\le& \C_{sq}^{2}\|x\|.
			\end{align*}
			\item \label{casec3} Otherwise, as in the proof of \ref{Case3} above, let $k_0:=\max_{k\in \N}\{n_k<|A|\}$, 
			choose $B\subset A$ so that $B\in \cG\left(x, n_{k_0},1\right)$, and let $m_1:=|A|-|B|$.  Let 
			\begin{align*}
				&a:=\min_{j\in A}|\be_j^*(x)|;&&b:=\min_{j\in B}|\be_j^*(x)|. 
			\end{align*}
			Note that $t\le a< b$. There are  two possibilities: \\
			\begin{enumerate}[\rm $\text{Case (ii)}$ .1.] 
				\item If \begin{align}
					\frac{a}{b}\le& \frac{1}{\vartheta_c\left(t\right)},\label{casi1}
				\end{align} 
				choose $D\subset B$ so that $D\in \cG\left(x,n_{k_0}-m_1,1\right)$, and let 
				\begin{align*}
					x_1:=&x-P_{B}(x)+\frac{a}{b}P_{D}(x). 
				\end{align*}
				Given that 
				\begin{align*}
					\left\vert \be_j^*(x_1)\right\vert=&\frac{a}{b} \left\vert \be_j^*(x)\right\vert\ge a \ge \left\vert \be_k^*(x)\right\vert=\left\vert \be_k^*(x_1)\right\vert=\forall j\in D\forall k\not \in A;\\
					\left\vert \be_j^*(x_1)\right\vert=& \left\vert \be_j^*(x)\right\vert\ge a \ge \left\vert \be_k^*(x)\right\vert=\left\vert \be_k^*(x_1)\right\vert=\forall j\in A\setminus B\forall k\not \in A,
				\end{align*}
				it follows that $\left(A\setminus B\right)\cup D\in \cG\left(x_1,n_{k_0},1\right)$. Hence,
				\begin{align*}
					\|x-P_A(x)\|=&\left\Vert x_1-P_{\left(A\setminus B\right)\cup D}(x_1)\right\Vert\le \C_{sq}\|x_1\|.
				\end{align*}
				By Remark~\ref{remarkwhyQ0}, there is $b\le b_1\le 1$ such that $D=A(x,b_1)$. Hence, by \eqref{casi1}, and Remark~\ref{remarkQGtheta} we have
				\begin{align*}
					\|x_1\|\le& \left\Vert x-P_B(x)\right\Vert +\frac{a}{b}\left\Vert P_{D}(x)\right\Vert\le \C_{sq}\|x\|+\frac{\vartheta\left(b_1\right)}{\vartheta_c\left(t\right)}\|x\|\le \C_{sq}\|x\|+\frac{\vartheta\left(t\right)}{\vartheta_c\left(t\right)}\|x\|\\
					\le&\left(\C_{sq}+2\right)\|x\|.
				\end{align*}
				Hence, 
				\begin{align}
					\|x-P_A(x)\|\le& \C_{sq}\left(\C_{sq}+2\right)\|x\|\le 3\C_{sq}^{2}\|x\|.    \label{casec3a}
				\end{align}
				\item If 
				\begin{align*}
					\frac{a}{b}>&\frac{1}{\vartheta_c\left(t\right)}, 
				\end{align*}
				let $x_1:=x-P_B(x)$. We have
				\begin{align*}
					&\|x_1\|\le \C_{sq}\|x\|; &&b^{-1}x_1\in \mathcal{Q}_0;&&&A\setminus B\in \cG\left(b^{-1}x_1, m_1,1\right),
				\end{align*}
				and 
				\begin{align*}
					\left\vert\be_j^*\left(b^{-1}x_1\right)\right\vert\ge \frac{a}{b}\ge \frac{1}{\vartheta_c\left(t\right)}\quad\forall j\in A\setminus B.
				\end{align*}
				By Remark~\ref{remarkwhyQ0}, there is $\left(\vartheta_c(t)\right)^{-1}\le t_1 \le 1$ such that $A\setminus B=A\left(b^{-1}x_1,t_1\right)$, which entails that
				\begin{align*}
					\left\Vert x_1-P_{A\setminus B}\left(x_1\right)\right\Vert=& b\left\Vert \frac{x_1}{b} -P_{A\setminus B}\left(\frac{x_1}{b}\right)\right\Vert  \le \vartheta_c\left(t_1\right)\|x_1\|\le \vartheta_c\left(\frac{1}{\vartheta_c\left(t\right)}\right)\|x_1\|. 
				\end{align*}
				Since 
				\begin{align*}
					x_1-P_{A\setminus B}\left(x_1\right)=x-P_A(x), 
				\end{align*}
				a combination of the above inequalities gives
				\begin{align}
					\|x-P_A(x)\|\le& \C_{sq}\vartheta_c\left(\frac{1}{\vartheta_c\left(t\right)}\right) \label{casec3b}.
				\end{align}
			\end{enumerate}
			As we have studied both possibilities for \ref{casec3}, from  \eqref{casec3b} with \eqref{casec3a} it follows that
			\begin{align*}
				\|x-P_A(x)\|\le& 3\C_{sq}^{2}\left(1+\vartheta_c\left(\frac{1}{\vartheta_c\left(t\right)}\right)\right). 
			\end{align*}
		\end{enumerate}
		Taking supremum and considering that $0<t<1$ is arbitrary, the estimates for \ref{casec1} and \ref{casec3} imply that
		\begin{align*}
			\vartheta_c\left(t\right)\le& 3\C_{sq}^2\left(1+\vartheta_c\left(\frac{1}{\vartheta_c\left(t\right)}\right)\right)\quad \forall 0<t<1,
		\end{align*}
		which together with \eqref{paso1} entails that
		\begin{align*}
			\vartheta_c\left(t\right)\le& 3\C_{sq}^2\left(1+d_1\left(1+\left(\log\left(\vartheta_c\left(t\right)\right)\right)^{d_1}\right)\right)\quad \forall 0<t<1.
		\end{align*}
		Thus,
		\begin{align*}
			1\le \frac{ 3\C_{sq}^2\left(1+d_1\left(1+\left(\log\left(\vartheta_c\left(t\right)\right)\right)^{d_1}\right)\right)}{\vartheta_c\left(t\right)}\xrightarrow[t\to 0]{}0, 
		\end{align*}
		a contradiction.

	\end{proof}
	
	A combination of Theorem~\ref{theoremnQGboundedgaps->QG!!!} and \cite[Proposition 3.1]{O2015} gives the following. 
	\begin{corollary}Let $\n$ be a sequence. The following are equivalent:
		\begin{itemize}
			\item $\n$ has arbitrarily large gaps. 
			\item There is a Markushevich basis that is $\n$-quasi-greedy but not quasi-greedy. 
			\item There is a Schauder basis that is $\n$-quasi-greedy but not quasi-greedy. 
		\end{itemize}
		
	\end{corollary}

	It is known that any quasi-greedy basis is $t$-quasi-greedy for all $0<t\le 1$ (\cite[Lemma 2.1]{DKO2015}), but it is an open question whether the corresponding implication holds for $\n$-QG bases (\cite[Question 4]{O2015}).  A partial answer was  given in \cite[Lemma 3.9]{BB}, whereas for $\n$ with bounded quotient gaps, an affirmative answer follows from \cite[Lemma 2.1]{DKO2015} and Theorem~\ref{theoremnQGboundedgaps->QG!!!}.  
	In the remainder of this section, we obtain a further partial result for $\n$ with arbitrarily large gaps, whereas another such result will be given in Section~\ref{sectionbidemocracy}.  We use the following definitions.

	\begin{definition}\cite{Wo}\label{uccc} We say that $\B$ is $\n$-unconditional for constant coefficients if there is $\C>0$ such that 
		\begin{eqnarray}\label{ucc}
			\|\bff_{\bfe A}\|\le \C\|\bff_{\bfe' A}\|
		\end{eqnarray}
		for all $A\subset \N$ with $|A|\in\n$ and all $\bfe, \bfe'\in \Psi_A$. The smallest constant verifying \eqref{ucc} is denoted by $\mathbf{K}_{u}$ and we say that $\mathcal B$ is $\K_u$-$\n$-unconditional for constant coefficients. If $\n=\mathbb N$, we say that $\mathcal B$ is $\mathbf{K}_u$-unconditional for constant coefficients.
	\end{definition}

	\begin{defi}\cite[Definition 3.1]{BB2}\label{definitionsuperdemo}
		We say that $\mathcal{B}$ is $\n$-superdemocratic if there exists a positive constant $\C$ such that 
		\begin{eqnarray}\label{demo}
			\Vert \mathbf{1}_{\varepsilon A}\Vert \leq \C\Vert \mathbf{1}_{\varepsilon' B}\Vert,
		\end{eqnarray}
		for all $A,B$ with $\vert A\vert\leq \vert B\vert$, $\vert A\vert,\vert B\vert\in\n$ and $\varepsilon\in\Psi_A, \varepsilon'\in\Psi_B$. The smallest constant verifying \eqref{demo} is denoted by $\Delta_s$ and we say that $\mathcal B$ is $\Delta_s$-$\n$-superdemocratic. If \eqref{demo} is satisfied for $\varepsilon\equiv\varepsilon'\equiv 1$, we say that $\mathcal B$ is $\Delta_d$-$\n$-democratic, where $\Delta_d$ is again the smallest constant for which the inequality holds. 
		If $\n=\mathbb N$, we say that $\B$ is $\Delta_d$-democratic and $\Delta_s$-superdemocratic.
	\end{defi}
	\begin{remark}\label{remarkUCC=SUCC}\rm Note that a basis is unconditional for constant coefficients if and only if it is suppression unconditional for constant coefficients (\cite[Remark 3.4]{BBG}), and a basis that is quasi-greedy for largest coefficients has these properties as well (\cite[Lemma 4.7]{AABW}). Also, a basis is superdemocratic if and only if it is democratic and unconditional for constant coefficients. 
	\end{remark}
	
	\begin{proposition}\label{proposition: dem+nqg->ntqg} Let  $\n$ be a sequence, and $\B$ be a basis for a Banach space $\X$ that is $\n$-quasi-greedy. If $\B$ quasi-greedy for largest coefficients and democratic, it is $\n$-$t$-quasi-greedy for all $0<t\le 1$.
	\end{proposition}
	\begin{proof}
		First suppose $\F=\R$. By \cite[Theorem 2.6]{AAB2023}, $\B$ is nearly unconditional, and by Remark~\ref{remarkUCC=SUCC}, it is superdemocratic. Let $\phi:(0,1]\rightarrow [1,\infty)$ and $\Lambda_{s}$ be $\B$'s near unconditionality function and superdemocracy constant respectively. \\
		Now fix $x\in \X$ and $A$ a $t$-greedy set for $x$ with $|A|\in \n$. We may assume that $A$ is not a greedy set for $x$, that $P_A(x)\not=0$ and $P_A(x)\not=x$. Let $B$ be a greedy set for $x$ with $|B|=|A|$. Set  $a:=\max_{i\in  B\setminus A}|\be_i^*(x)|$. Note that $a>0$ and, for every $i\in B\setminus A$ and $j\in A\setminus B$, 
		\begin{align}
			a\ge |\be_i^*(x)|\ge& |\be_j^*(x)\|\ge ta \ge t |\be_i^*(x)|.\label{proposition: dem+nqg->ntqg. General caseflattgreedt}
		\end{align}
		It follows from the above estimate that, for every $E\subset A\setminus B$, 
		\begin{align}
			\left\Vert P_{E}\left(x\right)\right\Vert=&a\left\Vert P_{E}\left(a^{-1}\left(x-P_{B}(x)\right)\right)\right\Vert\le a\phi(t)\left\Vert a^{-1}\left(x-P_{B}(x)\right)\right\Vert=\phi(t)\left\Vert x-P_{B}(x)\right\Vert\nonumber\\
			\le& \phi(t)\C_{sq}\|x\|. \label{proposition: dem+nqg->ntqg. General caseestimatetgreedy}
		\end{align}
		On the other hand, by convexity and superdemocracy, 
		\begin{align}
			\left\Vert P_{B\setminus A}\left(x\right)\right\Vert&\le \max_{\varepsilon\in\Psi_{B\setminus A}}a\|\bff_{\varepsilon, B\setminus A}\|\le a \Delta_s \|\bff_{\varepsilon(x), A\setminus B}\|. \label{proposition: dem+nqg->ntqg. General caseestimategreedy}
		\end{align}
		Choose $x^*\in \X^*$ with $\|x^*\|=1$ so that 
		\begin{align*}
			x^*\left(\bff_{\varepsilon(x), A\setminus B}\right)=&\|\bff_{\varepsilon(x), A\setminus B}\|,
		\end{align*}
		and let 
		\begin{align*}
			D:=&\left\lbrace j\in A\setminus B\colon \be_j^*(x)x^*\left(\be_n\right)\ge 0\right\rbrace.
		\end{align*}
		Using \eqref{proposition: dem+nqg->ntqg. General caseflattgreedt}  and \eqref{proposition: dem+nqg->ntqg. General caseestimatetgreedy} we obtain
		\begin{align*}
			a\|\bff_{\varepsilon(x), A\setminus B}\|=&ax^*\left(\bff_{\varepsilon(x), A\setminus B}\right)=a \sum_{n\in A\setminus B}\sgn\left(\be_n^*(x)\right)x^*\left(\be_n\right)\le a \sum_{n\in D}\sgn\left(\be_n^*(x)\right)x^*\left(\be_n\right)\\
			=& t^{-1}\sum_{n\in D}at \sgn\left(\be_n^*(x)\right)x^*\left(\be_n\right)\le t^{-1}\sum_{n\in D}\be_n^*(x) x^*\left(\be_n\right)\le t^{-1}\left\Vert P_D(x)\right\Vert \\
			\le&t^{-1} \phi(t)\C_{sq}\|x\|.
		\end{align*}
		Combining the above with \eqref{proposition: dem+nqg->ntqg. General caseestimategreedy} we get 
		\begin{align*}
			\left\Vert P_{B\setminus A}\left(x\right)\right\Vert\le& t^{-1}\phi(t)\Delta_s\C_{sq}\|x\|. 
		\end{align*}
		Thus,
		\begin{align*}
			\left\Vert P_{B\cap A}\left(x\right)\right\Vert\le& \left\Vert P_{B\setminus A}\left(x\right)\right\Vert+\|P_B(x)\|\le \left(t^{-1}\phi(t)\Delta_s\C_{sq}+\C_{q}\right)\|x\|. 
		\end{align*}
		Finally, a combination of the above estimate with \eqref{proposition: dem+nqg->ntqg. General caseestimatetgreedy} yields
		\begin{align}
			\|P_A(x)\|\le& \left\Vert P_{B\cap A}\left(x\right)\right\Vert+\left\Vert P_{A\setminus B}\left(x\right)\right\Vert\le \left( \C_{sq}\phi(t)\left(1+\frac{\Delta_s}{t}\right)+\C_q\right)\|x\|.\label{proposition: dem+nqg->ntqg. generalcasentgreedy}
		\end{align}
		This completes the proof of the case $\F=\R$.  If $\F=\mathbb{C}$, pick $x\in \X$ and $A$ as before.
		%
		% Using \cite[Corollary 2.3]{O2015}, we may assume that $x$ has finite support. 
		For each $i\in \N$, let 
		\begin{align*}
			(z_i;z_i^*):=&
			\begin{cases}
				(\sgn{\be_i^*(x)}\be_i,\sgn{\be_i^*(x)}^{-1}\be_i^*)& \text{ if }\be_i^*(x)\not=0,\\
				(\be_i,\be_i^*) & \text{otherwise.} 
			\end{cases}
		\end{align*}
		Since $|z_i^*(z)|=|\be_i^*(z)|$ for all $z\in \X$ and all $i\in \N$, $(z_i)_i$ is also an $\n$-quasi-greedy Markushevich basis for $\X$ with constant $\C_{q}$. Note that $z_i^*(x)=|\be_i^*(x)|$ for all $i\in \N$. Let
		$$
		\mathbb{T}:=\left\lbrace z\in \X: z_i^*(z)\in \R\;\forall i\in \N\right\rbrace.
		$$
		It is routine to check that $\mathbb{T}$ is a Banach space over $\R$ with the norm given by the restriction to $\mathbb{T}$ of the norm on $\X$, and that $(z_i)_{i\in \N}$ is a basis for $\mathbb{T}$ with biorthogonal functionals $(z_i^*\big|_{\mathbb{T}})_{i\in \N}$.  \\
		Since, for every for every $y\in \mathbb{T}$, the coordinates of $y$ with respect to  $(z_i)_{i\in \N}$ are the same as they are when we consider $y$ as an element of $\X$, we have
		$$
		\lim_{\substack{n\to +\infty\\n\in \n}}\|\G_{n}(y)-y\|_{\mathbb{T}}=\lim_{\substack{n\to +\infty\\n\in \n}}\|\G_{n}(y)-y\|=0
		$$
		for all $y\in \mathbb{T}$ and every choice of greedy approximations with $n\in \n$ for all $n$ , so $(z_i)_{i\in \N}$ is an $\n$-quasi-greedy basis for $\mathbb{T}$. For the same reason, the $\n$-quasi-greedy constant of $(z_i)_{i\in \N}$ as a basis for $\mathbb{T}$ is no greater than $\C_q$, and the $t$-greedy sets are also the same whether we consider $y$ as an element of $\X$ or of $\mathbb{T}$. Since $x\in \mathbb{T}$, we can apply the result for real Banach spaces to complete the proof. 
	\end{proof}

	\section{$\n$-bidemocracy}\label{sectionbidemocracy}
	\color{black}
	
	In 2003, S. J. Dilworth et al. (\cite{DKKT}) studied conditions under which the dual basis of a greedy (resp. almost greedy) basis is also greedy (resp. almost greedy). In this context, they introduced the notion of bidemocracy, which we extend in this section. First, we need the notion of the fundamental function of a basis: take $\mathbb Y$ as the subspace of $\mathbb X^*$ spanned by $\mathcal B^*$, and define  $$\mathbf{1}_{\varepsilon A}^*=\mathbf{1}_{\varepsilon A}^*[\mathcal B^*,\mathbb Y]:=\sum_{n\in A} \varepsilon_ne_n^*.$$

	We define the \emph{fundamental function} $\varphi$ \emph{of $\B$} and the \emph{fundamental function} $\varphi^*$ \emph{of $\B^*$} by
	$$\varphi(m)=\varphi[\mathcal B,\X](m):=\sup_{\vert A\vert\leq m, \vert\varepsilon\vert=1}\Vert\mathbf{1}_{\varepsilon A}\Vert,$$
	and
	$$\varphi^*(m)=\varphi[\mathcal B^*, \mathbb Y](m)=\sup_{\vert A\vert\leq m, \vert\varepsilon\vert=1}\Vert\mathbf{1}_{\varepsilon A}^*\Vert.$$

	Using these functions, we say that a basis is bidemocratic if $$\varphi(m)\varphi^*(m)\lesssim m,\, \forall m\in\N.$$
	
	\begin{remark}\rm
		Traditionally, the function $\varphi(m)$ has been defined using $\e\equiv 1$. Our definition is equivalent since 
		$$\sup_{\vert A\vert\leq m}\Vert\one_A\Vert\le \varphi(m) \leq 2\kappa \sup_{\vert A\vert\leq m}\Vert\one_A\Vert.$$
	\end{remark}

	The following result was proven in \cite{DKKT}. 
	
	\begin{theorem}\label{theoremDKKTbi}
		Let $\B$ be a quasi-greedy (resp. unconditional) basis. The following are equivalent:
		\begin{enumerate}[\rm \color{red}i)\color{black}]
			\item $\B$ is bidemocratic.
			\item $\B$ and $\B^*$ are both almost greedy (resp. greedy).
		\end{enumerate}
	\end{theorem}

	Here, we define the notion of $\n$-bidemocracy for any sequence $\n$, and study duality of $\n$-$t$-quasi greedy bases. 
	
	\begin{defi}
		We say that a basis $\B$ in a Banach space $\mathbb X$ is $\n$-bidemocratic if there exists a positive constant $\C$ such that
		\begin{eqnarray}\label{bi}
			\varphi(n)\varphi^*(n)\leq \C n,\; \forall n\in\n.
		\end{eqnarray}
		The smallest constant verifying \eqref{bi} is denoted by $\Delta_b$ and we say that $\B$  is $\Delta_b$-$\n$-bidemocratic. If $\n=\mathbb N$, we say that $\B$ is $\Delta_b$-bidemocratic.
	\end{defi}
	
	\begin{remark}\label{remarkdualbidem}\rm 
		Note that $\B^*$ is a basis for $\Y$, then $\B^{**}=(\hat{\be}_i\big|_{\Y})_{i\in \N}\subset \Y^*$, where $\hat{x}\in X^{**}$ is the image of $x\in \X$ via the canonical inclusion $X\hookrightarrow X^{**}$. Hence, $\varphi^{**}(m)\le \varphi(m)$ for all $m$. In particular, this implies that if $\B$ is $\Delta_b$-$\n$-bidemocratic, then $\B^*$ is $\C$-$\n$-bidemocratic with $\C\le \Delta_b$. 
	\end{remark}
	
	\begin{remark}\label{remarkbideml1} \rm Note that if $\B$ is $\n$-bidemocratic and $\varphi(n)\approx n$ for $n\in \n$, then $\varphi^*(n)$ is bounded for $n\in \n$, and so for all $n\in \N$. It follows that $\B^*$ is equivalent to the unit vector basis of $\mathtt{c}_{0}$, so $\B$ is equivalent to the unit vector basis of $\ell_1$ (see \cite[Corollary 5.7]{DKKT}).
	\end{remark}

	We  will show that if $\B$ is an $\n$-bidemocratic and $\n$-quasi-greedy basis, then the dual basis $\B^*$ is $\n$-$t$-quasi-greedy for all $0<t\le 1$. First we prove an auxiliary lemma. 
	
	\begin{lemma}\label{lemmanbidemqg2}Let $\B$ be a basis in a Banach space $\mathbb X$, and let $A$ and $B$ be nonempty sets of positive integers. If $B$ is a $t$-greedy set for $x^*\in \Y$, then 
		\begin{equation}
			|(x^*-P_Bx^*)(P_Ax)|\le t^{-1}\varphi^*(|A|)\varphi(|B|)|B|^{-1}\|x\|\|x^*\|\qquad\forall x\in \X.\label{partt12}
		\end{equation}
		Similarly, if $A$ is a $t$-greedy set for $x\in \X$, then 
		\begin{equation}
			|P_Bx^*(x-P_Ax)|\le t^{-1} \varphi^{*}(|A|)\varphi(|B|)|A|^{-1}\|x\|\|x^*\|\qquad \forall x^*\in \Y.\label{part22}
		\end{equation}
	\end{lemma}
	\begin{proof}
		To prove \eqref{partt12}, choose $\bfe' \in \Psi_B$ so that $x^*(\be_i)\varepsilon_i=|x^*(\be_i)|$ for all $i\in B$. Using convexity (see for instance \cite[Lemma 2.7]{BBG}), we obtain:
		\begin{eqnarray}
			|(x^*-P_Bx^*)(P_Ax)|&=&|\sum\limits_{i\in A\setminus B}x^*(\be_i)\be_i^*(x)|\le \|\sum\limits_{i\in A\setminus B}x^*(\be_i)\be_i^*\|\|x\|\nonumber\\
			&\le& \max_{i\in A\setminus B}|x^*(\be_i)|\varphi^*(|A|)\|x\|\le t^{-1}\min_{i\in B}|x^*(\be_i)|\varphi^*(|A|)\|x\|\nonumber\\
			&\le& t^{-1}|B|^{-1}\sum\limits_{i\in B}|x^*(\be_i)|\varphi^*(|A|)\|x\|\nonumber\\
			&=&t^{-1}|B|^{-1}|x^*(\bff_{\bfe' B})|\varphi^*(|A|)\|x\|\nonumber\\
			&\le& t^{-1}\varphi^*(|A|)\varphi(|B|)|B|^{-1}\|x\|\|x^*\|.\nonumber
		\end{eqnarray}
		
		This completes the proof of \eqref{partt12}. To prove \eqref{part22}, we use  \eqref{partt12} and Remark~\ref{remarkdualbidem} to obtain
		\begin{eqnarray*}
			|P_Bx^*(x-P_Ax)|&=&|(\hat{x}\big|_{\Y}-P_A\hat{x}\big|_{\Y})(P_Bx^*)|\le t^{-1} \varphi^*(|A|)\varphi^{**}(|B|)|A|^{-1}\|x^*\|\|\hat{x}\big|_{\Y}\|\\
			&\le& t^{-1} \varphi(|B|)\varphi^{*}(|A|)|A|^{-1}\|x^*\|\|x\|.
		\end{eqnarray*}
		
	\end{proof}

	\begin{proposition}\label{propositiondual2}
		Let $\B$ be a $\Delta_b$-$\n$-bidemocratic basis in a Banach space $\mathbb X$, and $0<s\le 1$. If $\B$ is $\C_{q,s}$-$\n$-$s$-quasi-greedy, then for all $0<t\le 1$, $\B^*$ is $\C_{q,t}^*$-$\n$-$t$-quasi-greedy, with
		
		\begin{equation*}
			\C_{q,t}\le \C_{q,s}+s^{-1}\Delta_b+t^{-1}\Delta_b.
		\end{equation*}
		
	\end{proposition}
	\begin{proof}
		Fix $x^*\in \Y$ with $\|x^*\|=1$, and $B$ a $t$-greedy set for $x^*$ with $|B|\in \n$. Given $x\in \X$ with $\|x\|=1$, choose $A$ an $s$-greedy set for $x$ so that $|A|=|B|$. Applying Lemma~\ref{lemmanbidemqg2} we obtain
		\begin{eqnarray*}
			|P_Bx^*(x)|&\le& |P_Bx^*(x-P_Ax)|+|x^*(P_Ax)|+|(x^*-P_Bx^*)(P_Ax)|\\
			&\le &s^{-1} \varphi^{*}(|A|)\varphi(|B|)|A|^{-1}+\C_{q,s}+t^{-1}\varphi^*(|A|)\varphi(|B|)|B|^{-1}\\
			&\le& \C_{q,s}+s^{-1}\Delta_b+t^{-1}\Delta_b. 
		\end{eqnarray*}
		As $x$ and $x^*$ are arbitrary, it follows that $\B^*$ is $\C^*_{q,t}$-$t$-quasi-greedy, with $\C^*_{q,t}$ as in the statement. 
	\end{proof}
	\begin{corollary}If $\B$ is an $\n$-bidemocratic, $\n$-quasi-greedy Schauder basis, it is $\n$-$t$-quasi-greedy for all $0<t\le 1$. 
	\end{corollary}
	\begin{proof}
		Since $\B$ is a Schauder basis, it is equivalent to $\B^{**}$ (see \cite{AK2016}*{Corollary 3.2.4}), so the result follows from Proposition~\ref{propositiondual2}. 
	\end{proof}

	Clearly, if $\B$ is bidemocratic, it is $\n$-bidemocratic for any $\n$. The converse is false for sequences with arbitrarily large quotient gaps, as Example~\ref{examplendemocratic} shows. On the other hand, for sequences with bounded quotient gaps we have the following result. 
	
	\begin{proposition}\label{propositionbilboundedv2}
		Let $\B$ be a basis and $\n$ a sequence with $l$-bounded quotient gaps. If $\B$ is $\n$-$\Delta_b$-bidemocratic, it is $\mathbf{B}$-bidemocratic with
		$$\mathbf{B}\leq\max\left\lbrace \alpha_1\alpha_2 (n_1-1), l \Delta_b\right\rbrace.$$
	\end{proposition}
	\begin{proof}
		Take $m\in\N$. If $m<n_1$, we have the following trivial bound:
		$$
		\varphi(m)\varphi^*(m)\leq \alpha_1\alpha_2 m^2\le \alpha_1\alpha_2 (n_1-1) m.
		$$
		Assume now that there exists $k\in \N$ such that $n_k \le m < n_{k+1}$.  Since $\n$ has $l$-bounded quotient gaps, we have
		\begin{equation*}
			\varphi(m)\varphi^*(m)\leq \varphi(n_{k+1})\varphi^*(n_{k+1})\le \Delta_b n_{k+1}\le l \Delta_b  n_{k}\le l \Delta_b  m.
		\end{equation*}
	\end{proof}

	It is proven in \cite[Proposition 4.2]{DKKT} that every bidemocratic basis is superdemocratic. Next, we extend that result to the case of $\n$-bidemocratic bases, and add the implications for $\n$-symmetry for largest coefficients and the $\n$-UL property.

	\begin{defi}\cite[Definition 4.6]{BB2}\label{definitionnslq}
		We say that a basis $\mathcal{B}$ in a Banach space $\mathbb X$ is $\n$-symmetric for largest coefficients if there exists a positive constant $\C$ such that 
		\begin{eqnarray}\label{sy}
			\Vert x+\mathbf{1}_{\varepsilon A}\Vert \leq \C\Vert x+\mathbf{1}_{\varepsilon' B}\Vert,
		\end{eqnarray}
		for any pair of sets $A,B$ with $\vert A\vert\leq \vert B\vert$, $A\cap B=\emptyset$, $\vert A\vert,\vert B\vert\in\n$, for any $\varepsilon\in\Psi_A, \varepsilon'\in\Psi_B$ and for any $x\in\X$ such that $\vert\be_i^*(x)\vert\leq 1\, \forall i\in\mathbb N$ and $\supp(x)\cap (A\cup B)=\emptyset$. The smallest constant verifying \eqref{sy} is denoted by $\Delta$ and we say that $\mathcal B$ is $\Delta$-$\n$-symmetric for largest coefficients. If $\n=\N$, we say that $\mathcal B$ is $\Delta$-symmetric for largest coefficients.
	\end{defi}

	\begin{lemma}\label{lemmanbidemnsuperdem} Let $\B$ be a $\Delta_b$-$\n$-bidemocratic basis in a Banach space $\mathbb X$. The following hold: 
	\end{lemma}
	\begin{enumerate}[\rm i)]
		\item \label{lemmabidemsuperdemi} $\B$ is $\Delta_s$-$\n$-superdemocratic, with $\Delta_s\le \Delta_b$. 
		\item \label{lemmabidemslcii}$\B$ is $\Delta$-$\n$-symmetric for largest coefficients, with $\Delta\le 1+2\Delta_b$. 
		\item  \label{lemmabidemULiii} $\B$ has the $\n$-UL property, with $\max\{\C_1,\C_2\}\le \Delta_b$.  
	\end{enumerate}
	\begin{proof}
		\ref{lemmabidemsuperdemi} Take $A, B$ such that $\vert A\vert=\vert B\vert\in\n$, $\e\in\Psi_A$ and $\e'\in\Psi_B$. Hence, if $\overline{\e'}$ is the conjugate of $\e'$,
		\begin{eqnarray*}
			\Vert\one_{\e A}\Vert \leq \Delta_b \dfrac{\vert B\vert}{\Vert \one_{\overline{\e'}B}^*\Vert}=\Delta_b\dfrac{\one_{\overline{\e'}B}^*(\one_{\e'B})}{\Vert \one_{\overline{\e'}B}^*\Vert}\leq \Delta_b \Vert \one_{\e' B}^*\Vert.
		\end{eqnarray*}
		Thus, by \cite[Remark 3.3]{BB2}, $\B$ is $\Delta_s$-superdemocratic with $\Delta_s\leq \Delta_b$.
		
		\ref{lemmabidemslcii} Take $A, B$ such that $\vert A\vert=\vert B\vert\in\n$, $x\in\X$ such that $\vert \be_n^*(x)\vert\leq 1$ $\forall n\in\N$, $A \cap B=\emptyset$, $\supp(x)\cap (A\cup B)=\emptyset$, $\e\in\Psi_A$ and $\e'\in\Psi_B$. Then,
		
		\begin{eqnarray}\label{bid1}
			\Vert x+\one_{\e A}\Vert \leq \Vert x+\one_{\e' B}\Vert+2\max\lbrace\Vert \one_{\e A}\Vert,\Vert\one_{\e' B}\Vert\rbrace.
		\end{eqnarray}
		Define now the element $y:=x+\one_{\e' B}$. Thus,
		\begin{eqnarray}\label{bid2}
			\nonumber\max\lbrace\Vert \one_{\e A}\Vert,\Vert\one_{\e' B}\Vert\rbrace&\leq& \min_{n\in B}\vert\be_n^*(y)\vert\varphi(\vert B\vert)\leq \Delta_b \min_{n\in B}\vert\be_n^*(y)\vert\dfrac{\vert B\vert}{\Vert\one_{\overline{\e'}B}^*\Vert}\\
			\nonumber&\leq& \Delta_b \dfrac{\sum_{n\in B}\vert\be_n^*(x+\one_{\e' B})\vert}{\Vert\one_{\overline{\e'}B}^*\Vert}=\Delta_b\dfrac{\one_{\overline{\e'}B}^*(x+\one_{\e' B})}{\Vert\one_{\overline{\e'}B}^*\Vert}\\
			&\leq& \Delta_b\Vert x+\one_{\e' B}\Vert.
		\end{eqnarray}
		
		By \eqref{bid1}, \eqref{bid2} and \cite[Lemma 4.7]{BB2}, $\B$ is $\n$-symmetric for largest coefficients with constant $\Delta\leq 1+2\Delta_b$.
		
		\ref{lemmabidemULiii} Consider a sequence of scalars $(a_n)_{n\in A}$ with $A$ a finite set with $\vert A\vert\in\n$. On the one hand,
		\begin{eqnarray}
			\Vert\sum_{n\in A}a_n\be_n\Vert \stackrel{\text{convexity}}{\leq}\max_{n\in A}\vert a_n\vert\varphi(\vert A\vert)\stackrel{\ref{lemmabidemsuperdemi}}{\leq}\Delta_b \max_{n\in A}\vert a_n\vert\Vert\one_A\Vert.
		\end{eqnarray}
		
		To show that 
		\begin{eqnarray}
			\min_{n\in A}\vert a_n\vert\Vert \one_A\Vert \leq \Delta_b\Vert \sum_{n\in A}a_n \be_n\Vert,
		\end{eqnarray}
		we only have to repeat the argument used to show \eqref{bid2}. Then, $\mathcal B$ has the $\n$-UL property with constants $\lbrace \mathbf{C}_1, \mathbf{C}_2\rbrace\leq \Delta_b$.
	\end{proof}
	
	For $\n$ with bounded quotient gaps, one can use the $\n$-bidemocracy constant of an $\n$-QG basis to estimate the quasi-greedy and superdemocracy constants. 
	
	\begin{corollary}\label{corollaryboundedbidemnqg->qg}Let $\B$ be a basis, and $\n$ a sequence with $l$-bounded quotient gaps. If $\B$ is $\Delta_d$-$\n$-bidemocratic and $\C_{q,t}$-$\n$-$t$-quasi-greedy, it is $\C$-$t$-quasi-greedy with 
		$$
		\C\le \max\{\alpha_1 \alpha_2 (n_1-1),  \C_{q,t}\left(1+\left(l-1\right)\Delta_b^2\right) \},  
		$$
		and is $\mathbf{M}$-superdemocratic with 
		$$
		\mathbf{M}\le \max\left\lbrace \alpha_1\alpha_2 (n_1-1), l \Delta_b\right\rbrace.
		$$
	\end{corollary}
	\begin{proof}
		By Lemma~\ref{lemmanbidemnsuperdem}, $\B$ is $\Delta_s$-$\n$-superdemocratic and has the $\n$-UL property with constants $\C_1$ and $\C_2$ such that 
		$$
		\max\{\C_1,\C_2, \Delta_s\}\le \Delta_b.
		$$
		Then, by \cite[Proposition 4.14]{BB2}, $\B$ is $\C$-$t$-quasi-greedy with 
		$$
		\C\le \max\{\alpha_1\alpha_2 (n_1-1), \C_{q,t}\left(1+\left(l-1\right)\Delta_b^2\right)\}. 
		$$
		By Proposition~\ref{propositionbilboundedv2} and Lemma~\ref{lemmanbidemnsuperdem}, $\B$ is $\mathbf{M}$-superdemocratic, with $\mathbf{M}$ as in the statement.
	\end{proof}

	\section{Some $\n$-greedy-type bases}\label{sectionuncond}
	
	\color{black}

	In greedy approximation theory, there are several ways to study the convergence of the TGA. For instance, for quasi-greedy bases the algorithm converges, but we do not know how fast it does. To study other types of convergence, we can consider among others greedy bases (\cite{KT}), almost greedy bases (\cite{DKKT}), semi-greedy bases (\cite{DKK2003}), partially greedy bases (\cite{DKKT}) and strong partially greedy bases (\cite{BBL}). 
	Here, we study the extensions of some of these notions - as well as some closely related ones. In \cite{O2015}, the two following extensions were considered.
	
	\begin{defi}
		We say that a basis $\B$ in a Banach space $\mathbb X$ is $\n$-greedy if there exists a positive constant $\C$ such that
		\begin{eqnarray}\label{greedy}
			\Vert x-\G_n(x)\Vert \leq \C\inf_{\vert\supp(y)\vert\leq n}\Vert x-y\Vert,\; \forall x\in\X, \forall \G_n\in\mathcal{G}_n, \forall n\in\n.
		\end{eqnarray}
	\end{defi}
	
	\begin{defi}
		We say that a basis $\B$  in a Banach space $\mathbb X$ is $\n$-almost greedy if there exists a positive constant $\C$ such that
		\begin{eqnarray}\label{agreedy}
			\Vert x-\G_n(x)\Vert \leq \C\inf_{\vert A\vert\leq n}\Vert x-P_A(x)\Vert,\; \forall x\in\X, \forall \G_n\in\mathcal{G}_n, \forall n\in\n.
		\end{eqnarray}
	\end{defi}
	
	If $\n=\N$, we recover the classical definition of greedy and almost-greedy bases. One interesting result is in \cite[Remark 1.1]{O2015}, where the author proved that for any sequence $\n$, $\n$-greediness (resp. $\n$-almost greediness) is equivalent to greediness (resp. almost greediness) and this fact does not happen for $\n$-quasi-greedy bases as we have mentioned at the beginning of the paper.
	
	\begin{remark}\label{remarkngreedy}\rm
		Although we have considered $\n$-greediness and $\n$-almost greediness for $t=1$, we can extend this version to the WTGA, and combining \cite[Remark 1.1]{O2015} with \cite[Theorem 1.5.1]{Temlyakov2008} (resp. \cite[Theorem 1.5.4]{Temlyakov2008}), we can obtain that for any basis $\B$, the following are equivalent.
		\begin{enumerate}[\rm i)]
			\item $\B$ is $\n$-$t$-greedy (resp. $\n$-$t$-almost greedy).
			\item $\B$ is $t$-greedy (resp. $t$-almost greedy) 
			\item $\B$ is greedy (resp. almost greedy) .
			\item $\B$ is $\n$-greedy (resp. $\n$-almost greedy). 
		\end{enumerate}
	\end{remark}

	As Oikhberg also proved that the $\n$-quasi-greedy property is not equivalent to the quasi-greedy property (\cite[Proposition 3.1]{O2015}), it is natural to ask whether equivalence holds for other intermediate properties. For example, per \cite{KT}, it is known that a basis is greedy if and only if it is unconditional and democratic. Thus, for seminormalized bases, the notion of unconditionality lies between those of quasi-greediness and greediness. Does $\n$-unconditionality  - defined in a natural manner - entail unconditionality? 
	Similarly, in this context we can ask whether the equivalence between greediness on one hand and unconditionality plus democracy on the other hand, holds for the respective extensions, namely $\n$-greediness, $\n$-unconditionality, and $\n$-democracy. Similar questions arise for other greedy-like properties. This section is dedicated to the study of some of these questions. We begin with the following definition, which extends the notion of unconditionality - or more precisely, the equivalent notion of suppression unconditionality - to the context of sequences with gaps. 
	\begin{defi}
		We say that a basis $\B$ in a Banach space $\mathbb X$ is $\n$-suppression unconditional if there exists a positive constant $\C$ such that
		\begin{eqnarray}\label{unc}
			\Vert P_A(x)\Vert\leq \C\Vert x\Vert,\; \forall x\in\X, \forall A\subset \N: \vert A\vert\in\n.
		\end{eqnarray}
		The smallest constant verifying \eqref{unc} is denoted by $\K_s$ and we say that $\B$ is $\K_s$-$\n$-suppression unconditional. If $\n=\mathbb N$, we say that $\B$ is $\K_s$-suppression unconditional.
	\end{defi}
	
	It is immediate that suppression unconditionality entails $\n$-suppression unconditionality. It turns out that the reverse implication holds as well. 
	\begin{proposition}\label{uncon}
		Let $\B$ be a $\K_s$-$\n$-suppression unconditional basis in a Banach space $\mathbb X$. Then, $\B$ is $\K_s$-suppression unconditional.
	\end{proposition}
	\begin{proof}
		Let $A\subset\N$ be a finite set and $x\in\X$ with finite support. Define the element $z:=x+\e\one_C$, where $C>\supp(x)$ is such that $\vert A\vert+\vert C\vert\in\n$. We have
		$$\Vert P_A(x)\Vert \leq \Vert P_{A\cup C}(z)\Vert+\e\Vert\one_C\Vert\leq \K_s\Vert x\Vert+(1+\K_s)\e\Vert\one_C\Vert.$$
		Letting $\e\rightarrow 0$ and using the density of elements of finite support, we conclude that the basis is $\K_s$-suppression unconditional.
	\end{proof}
	
	Combining \ref{uncon} and the aforementioned results from \cite{O2015}, it follows that a basis is $\n$-greedy if and only if it is $\n$-unconditional and democratic. Now, can we replace democracy by $\n$-democracy? It turns out the answer is negative in general. In fact, we have the following result. 
	\begin{proposition}\label{propositionnonngreedy}Suppose $\n$ has arbitrarily large quotient gaps. Then, there is a Banach space $\X$ with a basis $\B$ that is unconditional and $\n$-superdemocratic, but not democratic and thus, not greedy. 
	\end{proposition}
	\begin{proof}
		See Remark~\ref{remarkunconditional}.
	\end{proof}
	On the other hand, equivalence does hold for sequences with bounded quotient gaps. 
	
	\begin{theorem}\label{thgreedy}
		Suppose $\n$ has bounded quotient gaps. Then a basis $\B$ is $\n$-greedy if and only if $\B$ is $\n$-unconditional and $\n$-democratic.
	\end{theorem}
	\begin{proof}
		If $\B$ is $\n$-greedy, by \cite[Remark 1.1]{O2015} it is greedy and then, using the main result of \cite{KT}, it is unconditional and democratic. Now, if $\B$ is $\n$-unconditional and $\n$-democratic, by \cite[Proposition 3.6]{BB2} and Proposition~\ref{uncon}, the basis is unconditional and democratic. Thus, it is greedy, or equivalently $\n$-greedy.
	\end{proof}
	
	Similar results hold for the property of being almost greedy and the usual characterization in terms of quasi-greediness and democracy or superdemocracy. 
	\begin{proposition}\label{propositionnonalmogstgreedy}Suppose $\n$ has arbitrarily large quotient gaps. Then, there is a Banach space $\X$ with a Schauder basis $\B$ that is $\n$-quasi-greedy and $\n$-superdemocratic, but neither democratic nor quasi-greedy.
	\end{proposition}
	\begin{proof}
		See Example~\ref{examplendemocratic}.
	\end{proof}
	
	\begin{theorem}\label{thalmost}
		Suppose $\n$ has bounded quotient gaps. Then a Schauder basis $\B$ is $\n$-almost-greedy if and only if it is $\n$-quasi-greedy and $\n$-democratic.
	\end{theorem}
	
	\begin{proof}
		If $\B$ is $\n$-almost greedy,  by \cite[Remark 1.1]{O2015} it is almost greedy and then, using the characterization of these bases proved in \cite{DKKT}, it is quasi-greedy and democratic. Now, if $\B$ is $\n$-quasi-greedy and $\n$-democratic, by \cite[Theorem 5.2]{BB} and \cite[Proposition 3.6]{BB2}, it is quasi-greedy and democratic and then almost greedy and $\n$-almost greedy.
	\end{proof}

	\section{$\n$-semi-greedy bases.} \label{sectionsemi}
	\color{black}
	We turn now our attention to the semi-greedy property, extended to our context. In order to define $\n$-$t$-semi-greedy bases, we could extend the definicion of $t$-weak-semi-greedy bases from \cite{BL2020} to the context of sequences with gaps, or give a definition in line with those of $\n$-$t$-quasi-greedy, $\n$-$t$-almost greedy and $\n$-$t$-greedy bases from \cite{O2015}. Given our context, we choose the latter option for our definition, but we will also give results under hypotheses that are an extension of the former.

	\begin{definition}\label{definitionnsemigreedy}
		We say that a basis $\B$ in a Banach space $\mathbb X$ is $\n$-$t$-semi-greedy if there exists a positive constant $\C$ such that for all $n\in \n$, $x\in \X$, and $A$ any $t$-greedy set for $x$ of cardinaltity $n$, there is $z\in [\be_i: i\in A]$ such that
		\begin{eqnarray}\label{sgreedy}
			\| x-z\|\le  \C \inf_{\substack{|\supp{(y)}|\le n\\ y=P_{\supp(y)}(y)}} \|x-y\|.
		\end{eqnarray}
		The $\n$-$t$-semi-greedy constant of the basis $\C_{sg,t}$ is the minimum $\C$ for which the above inequality holds. 
	\end{definition}

	If $\n=\N$, we say that $\mathcal B$ is $\C_{sg,t}$-$t$-semi-greedy and, if in addition $t=1$, we recover the classical definition of semi-greedy bases from \cite{DKK2003}. It is known that in this case, for Markushevich bases the semi-greedy property is equivalent to the almost greedy property (\cite{BL2020}, \cite{B2019}, \cite{DKK2003}). As the $t$-almost greedy and $\n$-$t$-almost greedy properties are equivalent, it is natural to ask whether the $\n$-$t$-semi-greedy property is also equivalent to the $t$-semi-greedy property, and also to the $s$-semi-greedy property for all $0<s\le 1$. To tackle the case of Markushevich bases, we will consider the following separation property. 
	
	\begin{definition}\cite[Definition 3.1]{BL2020}\label{definitionseparation}Let $(u_i)_i\subset \X$ be a sequence in a Banach space. We say that $(u_i)_i$ has the \emph{finite dimensional separation property} (or FDSP) if there is a positive constant $\C$ such that for every separable subspace $\mathbb{L}\subset \X$ and every $\epsilon>0$, there is a basic subsequence $(u_{i_k})_k$ with basis constant no greater than $\C+\epsilon$ and the following property: for every finite dimensional subspace $\mathbb{E}\subset \mathbb{L}$ there is $j_{\mathbb{E}}\in \N$ such that 
		\begin{equation}
			\|x\|\le (\C+\epsilon)\|x+z\|.\label{separation}
		\end{equation}
		for all $x\in \mathbb{E}$ and all $z\in \overline{[u_{i_k}:k> j_{\mathbb{E}}]}$. Any such subsequence is called a \emph{finite dimensional separating sequence} for $(\mathbb{L}, \C, \epsilon)$, and the minimum $\C$ for which this property holds is \emph{finite dimensional separation constant} $M_{fs}$ of $(u_i)_i$. 
	\end{definition}
	It is known that every Markushevich basis has the FDSP property (see  \cite[Proposition 3.11]{BL2020} for this result and estimates for the constant). We will also the following result, which is part of \cite[Theorem 7.1]{DKSWo2012}, restated for our purposes. 
	
	\begin{theorem}\label{theorem71} Let $\B$ be an almost greedy basis with quasi-greedy constant $\C_{q}$ and democracy constant $\Delta_d$. Then for every $0<t\le 1$, $\B$ is $t$-semi-greedy with constant $\C_{sg,t}$ that only depends on $t$, $\C_{q}$ and $\Delta_d$.
	\end{theorem}
	
	Now we can prove that $\n$-$t$-semi-greedy Markushevich bases are semi-greedy.

	\begin{lemma}\label{lemmasemigreedynsemigreedy}Let $\B$ be an $\C_{sg,t}$-$\n$-$t$-semi-greedy Markushevich basis in a Banach space $\mathbb X$. Then, $\B$ is $\K_{sg,t}$-$t$-semi greedy with $\K_{sg,t}\le \C_{sg,t}M_{fs}$. Moreover, it is $s$-semi greedy for all $0<s\le 1$, with constants only depending on $s$, $\C_{sg,t}$ and $M_{fs}$. 
	\end{lemma}
	\begin{proof}
		First we prove that $\B$ is $t$-semi greedy, with constant as in the statement. \\
		Fix $\epsilon>0$, and let $(\be_{i_j})_j$ be a separating sequence for $(\X,M_{fs},\epsilon)$. Choose $x\in \X\setminus \{0\}$, $m\in \N$, and $A$ an $m$-$t$-greedy set for $x$. If $|\supp(x)|\le m$, then $\supp(x)\subset A$ and then $x=P_{A}(x)$ because $\B$ is a Markushevich basis. Thus, 
		$$
		\|x-P_{A}(x)\|=0= \inf_{\substack{|\supp{(y)}|\le m }}\|x-y\|,
		$$
		and we are done. On the other hand, if  $|\supp(x)|> m$, then 
		$$
		\inf_{\substack{|\supp{(y)}|\le m }}\|x-y\|>0
		$$
		(see, for example, \cite[Lemma 4.8]{BL2020}), so one can choose $\delta>0$ and $y_0\in \X$ with $|\supp{(y_0)}|\le m$ so that 
		\begin{equation}
			\|x-y_0\|\le (1+\delta)+\inf_{\substack{|\supp{(y)}|\le m}}\|x-y\|.\label{closeenoughsg}
		\end{equation}
		Now pick $n\in \n$ so that $n>m$, let $a:=\sup_{i\in \N}|\be_i^*(x)|$, and define 
		\begin{equation*}
			\mathbb{E}:=[x, \be_i: i\in A ], \quad B:=\{i_{j_{\mathbb{E}}+1}, \dots, i_{j_{\mathbb{E}}+n-m}\},\quad \text{and}\quad v:=x+(2at^{-1}+1)\bff_{B}.
		\end{equation*}
		Since
		\begin{equation*}
			t |\be_i^*(v)|>a >|\be_j^*(v)|\quad \forall i\in B\;\forall j\not\in B,
		\end{equation*}
		we have that  $A\cupdot B$ is an $n$-$t$-greedy set for $v$. Hence, by the $\n$-$t$-semi-greedy condition there is $z\in \X$ with $\supp{(z)}\subset A\cupdot B$ such that
		\begin{equation*}
			\|v-z\|\le \C_{sg,t}\|v-(y_0+(2at^{-1}+1)\bff_{B})\|=\C_{sg,t}\|x-y_0\|.% \label{nsg1}
		\end{equation*}
		From this and \eqref{closeenoughsg}, applying the separation condition we obtain
		\begin{align}
			\|x-P_{A}(z)\|&\le (M_{fs}+\epsilon)\|x-P_{A}(z)+(2at^{-1}+1)\bff_{B}-P_{B}(z)\|=(M_{fs}+\epsilon)\|v-z\|\nonumber\\
			&\le (M_{fs}+\epsilon)\C_{sg,t}\|x-y_0\|\le (M_{fs}+\epsilon)\C_{sg,t}(1+\delta)\inf_{|\supp{(y)}|\le m}\|x-y\|). \nonumber
		\end{align}
		Since $A$ is a finite set and $\epsilon, \delta$ are arbitrary, it follows that 
		$$
		\min_{\supp{(u)}\subset A}\|x-u\|\le \C_{sg,t}M_{fs}\inf_{\substack{|\supp{(y)}|\le m}}\|x-y\|.
		$$
		This proves that $\B$ is $\K_{sg,t}$-$t$-semi greedy with $\K_{sg,t}\le \C_{sg,t}M_{fs}$. In particular, it is semi-greedy with constant no greater than $\C_{sg,t}M_{fs}$. Hence, by \cite[Theorem 4.2]{BL2020}, it is almost greedy with democratic and quasi-greedy constants depending only on $\C_{sg,t}$ and $M_{fs}$. Now Theorem~\ref{theorem71} entails that for all $0<s\le 1$, $\B$ is $s$-semi greedy, with constant as in the statement. 
	\end{proof}
	
	\begin{remark}\rm If $\B$ has a weakly null subsequence, then by \cite[Proposition 3.11]{BL2020} we have $M_{fs}=1$, so Lemma~\ref{lemmasemigreedynsemigreedy} gives $\K_{sg,t}=\C_{sg,t}$. Similarly, if $\B$ is Schauder with basis constant $\K$, then $M_{fs}\le \K$, so  we get $\K_{sg,t}\le \C_{sg,t}\K$. \end{remark}
	
	We can relax the hypotheses of  Lemma~\ref{lemmasemigreedynsemigreedy} and still obtain that $\B$ is $t$-semi-greedy, though we do not get the same bound for the constant.  The proof is very similar to that of Lemma~\ref{lemmasemigreedynsemigreedy} - with some   differences due to the fact that we may not choose the $t$-greedy set, so we shall be brief.

	\begin{lemma}\label{lemmastronger}Let $0<t\le 1$ and $\C>0$. Suppose $\B$ is a Markushevich basis in a Banach space $\mathbb X$ with the property that for all $n\in \n$ and all $x\in \X$, there is $z\in \X$ with support in a $t$-greedy for $x$ of cardinality $n$ such that 
		\begin{align*}
			\| x-z\|\le&  \C\inf_{\substack{|\supp{(y)}|\le n}} \|x-y\|.
		\end{align*}
		Then, for all $x\in \X$ and all $m\in \N$, there is a $t$-greedy set $A$ for $x$ with $|A|=m$ and  $z\in \X$ with $\supp{(z)}\subset A$ such that
		\begin{align}
			\| x-z\|\le  \C M_{fs} \inf_{|\supp{(y)}|\le m} \|x-y\|.\label{generalsemi}
		\end{align}
		Hence, $\B$ is almost greedy and, for all $0<s\le 1$, $\B$ is $s$-semi greedy, with constant only depending on $s$, $\C$, and $M_{fs}$. \\
	\end{lemma}
	%Note that if $\B$ is $\C_{sg,t}$-$\n$-$t$-semi-greedy, the above conditions hold with $\C=\C_{sg,t}$. 
	\begin{proof}
		Fix $\epsilon>0$, and let $(\be_{i_j})_j$ be a separating sequence for $(\X,M_{fs},\epsilon)$. Fix $x\in \X\setminus\{0\}$ with finite support $D$, and $m\in \N$. If $|D|\le m$, then $x=P_{D}(x)$ because $\B$ is a Markushevich basis. Thus, 
		$$
		\|x-P_D(x\|=0= \inf_{\substack{|\supp{(y)}|\le m }}\|x-y\|,
		$$
		and we are done. On the other hand, if $|D|> m$, as in the proof of Lemma~\ref{lemmasemigreedynsemigreedy}, given $\delta>0$ one can choose $y_0\in \X$ with $|\supp(y_0)|\le m$  so that 
		\begin{align}
			\|x-y_0\|\le& (1+\delta)\inf_{\substack{|\supp{(y)}|\le m }}\|x-y\|.\nonumber
		\end{align}	
		Now pick $n\in \n$ so that $n>m$, let $a:=\sup_{i\in \N}|\be_i^*(x)|$, and define 
		\begin{equation*}
			\mathbb{E}:=[\be_i: i\in D], \quad B:=\{i_{j_{\mathbb{E}}+1}, \dots, i_{j_{\mathbb{E}}+n-m}\},\quad \text{and}\quad v:=x+(at^{-1}+1)\bff_{B},
		\end{equation*}
		Note that $|\supp(v)|=|D \cupdot B|=|D|+n-m>n$. Let $A$ be a $t$-greedy set for $v$ with $|A|=n$ and $z\in \X$ with $\supp(z)\subset A$ such that 
		\begin{align}
			\| v-z\|\le&  \C\inf_{\substack{|\supp{(y)}|\le n}} \|v-y\|.\label{anothersg2}
		\end{align}
		Since
		\begin{equation*}
			\supp{(x)}\cap B=\emptyset \quad\text{and}\quad t |\be_i^*(v)|>a \ge |\be_j^*(v)|\quad \forall i\in B\;\forall j\not\in B,
		\end{equation*}
		it follows that $B\subset A$, and then $A\setminus B\subset D$ is a $t$-greedy set for $x$ of cardinality $m$. Applying the separation condition and \eqref{anothersg2}, we get 
		\begin{align*}
			\|x-P_{A\setminus B}(z)\|&\le (M_{fs}+\epsilon)\|x-P_{A\setminus B}(z)+(at^{-1}+1)\bff_{B}-P_{B}(z)\|=(M_{fs}+\epsilon)\|v-z\|\nonumber\\
			&\le (M_{fs}+\epsilon)\|v-y_0-(at^{-1}+1)\bff_{B} \|= (M_{fs}+\epsilon)\C\|x_0-y_0\|\\
			&\le (M_{fs}+\epsilon)(1+\delta)\C\inf_{|\supp{(y)}|\le m}\|x-y\|. \nonumber
		\end{align*}
		Given that $\delta$ and $\epsilon$ are arbitrary and there are only finitely many $t$-greedy sets for $x$ of cardinality $m$, we conclude that there is a $t$-greedy set $A_2$ for $x$ with $|A_2|=m$ and $z_2$ supported in $A_2$ such that
		\begin{align*}
			\|x-z_2\|\le M_{fs}\C \inf_{|\supp{(y)}|\le m}\|x-y\|.
		\end{align*}
		A density argument extends the result to vectors with infinity support. Now \cite[Theorem 4.2]{BL2020} entails that $\B$ is quasi-greedy and superdemocratic with contants depending only on $\C$ and $M_{fs}$, and the proof is completed by an application of Theorem~\ref{theorem71}.
	\end{proof}

	\begin{remark}\rm \label{remarkFDSPmark} Note that in Lemmas~\ref{lemmasemigreedynsemigreedy} and~\ref{lemmastronger}, the Markushevich hypothesis is only used to guarantee that $\B$ has the FDSP. Hence, if we replace the Markushevich hypothesis in those results by the hypothesis that $\B$ has the FDSP, we still obtain that $\B$ is almost greedy, and thus a Markushevich basis.
	\end{remark}
	
	To finish our study of $\n$-$t$-semi-greedy bases, we will consider the case where we remove the Markushevich hypothesis. By Remark~\ref{remarkFDSPmark} and  \cite[Corollary 3.9]{BL2020}, we get the following equivalence, which we will use in our next result. 
	\begin{lemma}\label{lemmanotmark} Let $\B$ be an $\n$-semi-greedy basis. The following are equivant. 
		\begin{itemize}
			\item $\B$ is not a Markushevich basis. 
			\item $\B$ does not have the FDSP. 
			\item The set $\{\be_i\}_{i\in \N}$ is weakly compact, and $0\not \in \overline{\{\be_i\}_{i\in \N}}^{w}$. 
		\end{itemize}
	\end{lemma}
	
	Now we can prove a result for $n$-$t$-semi-greedy bases that are not Markushevich bases.  The proof is a modification of the proof of  \cite[Proposition 4.9]{BL2020}, though we give a proof for the sake of completion. 
	
	\begin{lemma}\label{lemmanotmark2}
		Let $0<t\le 1$ and $\K>0$. Suppose $\B$ is a basis in a Banach space $\mathbb X$ with the property that for all $n\in \n$ and all $x\in \X$, there is a $t$ greedy set $A$ for $x$ with $|A|=n$ and $z\in [\be_i: i\in A]$ such that
		\begin{align*}
			\| x-z\|\le&  \K\inf_{\substack{y\in \X\\|\supp_{\B}{(y)}|\le n\\ y=P_{\B,\supp_{\B}(y)}(z) }} \|x-y\|.
		\end{align*}
		If $\B$ is not a Markushevich basis, there is $x_0\in \X$ and $x_0^*$ in $\X^*$ such that 
		$$
		\B_2:=\left(x_0,\be_i-x_0^*\left(\be_i\right)x_0\right)_{i \in \N}
		$$
		is an almost greedy basis for $\X$ with dual basis
		$$
		\B_2^*:=\left(x_0^*, \be_i^*\right)_{i\in \N}.
		$$
		The almost greedy constant of $\B_2$ depends only on $t$ and $\K$. 
	\end{lemma}
	Note that the conditions above hold if $\B$ is $\C_{sg,t}$-$\n$-$t$-semi-greedy, with $\K\le \C_{sg,t}$. 
	\begin{proof}
		By Lemma~\ref{lemmanotmark}, there is $x_1\in \X\setminus \{0\}$ and a subsequence $\left(\be_{i_l}\right)_{l\in \N}$ that converges weakly to $x_1$. By the Hahn-Banach Theorem, there is $x_1^*\in\X^*$ such that 
		$$
		\|x_1\|\|x_1^*\|=x_1^*(x_1)=1. 
		$$
		Let 
		$$
		a_0:=\frac{1}{\|x_1\|}\sup_{i\in \N}\left \Vert\be_i-x_1^*\left(\be_i\right)x_1\right\Vert,\quad x_0:=a_0x_1,\quad x_0^*:=\frac{x_1^*}{a_0}.
		$$
		Then 
		\begin{align}
			\|x_0\|\|x_0^*\|=& x_0^*(x_0)=1, \qquad \|x_0\|=\sup_{i\in \N}\left \Vert\be_i-x_0^*\left(\be_i\right)x_0 \right\Vert\label{norms}
		\end{align}
		and
		\begin{align}
			\be_{i_l}-x_0^*\left(\be_{i_l}\right)x_0\xrightarrow[l\to \infty]{w}0. \label{weaklynull}
		\end{align}
		Let $\B_1:=\left(\be_{i}-x_0^*\left(\be_{i}\right)x_0\right)_{i\in \N}$ and $\Z:=\overline{\left[\B_1\right]}$. Then $\B_1$ is a basis for $\Z$ with dual basis $\B_1^*:=\B^*\big|_{\Z}$. Note that \eqref{weaklynull} and \cite[Lemma 3.5, Remark 3.6]{BL2020} entail that $\B_1$ has the FDSP with $M_{fs}=1$. \\
		We will prove that we can apply Lemma~\ref{lemmastronger} to $\B_1$.  To that end, first define $T:\X\rightarrow \Z$ by
		$$
		T\left(x\right):=x-x_0^*\left(x\right)x_0. 
		$$
		Then $T$ is a bounded projection with $\|T\|\le 2$ and $T\left(\X\right)=\Z$. Fix $0<\epsilon<1$, $z\in \Z$  and $n\in \N$. Let $D:=\supp_{\B_1}(z)$. If 
		$$
		\inf_{\substack{y\in \Z\\|\supp_{\B_1}(y)|\le n\\y=P_{\B_1,\supp(y)}(y)}}\|z-y\|=0, 
		$$
		then by \cite[Lemma 4.8]{BL2020}, $z=P_{\B_1, D}(z)$ and $|D|\le n$. 
		Hence, 
		\begin{align}
			\left\Vert z-P_{D}(z)\right\Vert=0=\inf_{\substack{y\in \Z\\|\supp_{\B_1}(y)|\le n\\y=P_{\B_1,\supp(y)}(y)}}\|z-y\|.\label{smallsupport}
		\end{align}
		On the other hand, if 
		$$
		\inf_{\substack{y\in \Z\\|\supp_{\B_1}(y)|\le n\\y=P_{\B_1,\supp(y)}(y)}}\|z-y\|>0, 
		$$
		then $|\supp_{\B_1}(z)|>n$ and there is $A\subset \N$ with $|A|= n$ and scalars $(a_i)_{i\in A}$ such that 
		\begin{align}
			\left\Vert z-\sum_{i\in A}a_i\left(\be_i-x_0^*\left(\be_i\right)x_0\right)\right\Vert \le (1+\epsilon)\inf_{\substack{y\in \Z\\|\supp_{\B_1}(y)|\le n\\y=P_{\B_1,\supp_{\B_1}(y)}(y)}}\|z-y\|.\nonumber%\label{approx1sg}
		\end{align}
		Now let 
		$$
		u:=z+\sum_{i\in A}a_ix_0^*\left(\be_i\right)x_0. 
		$$
		By hypothesis, there is a set $B$ with $|B|=n$ that is $t$-greedy for $z$ with respect to $\B$, and scalars $(b_j)_{j\in B}$ such that 
		\begin{align}
			\left\Vert u-\sum_{j\in B}b_j\be_j\right\Vert\le \K\inf_{\substack{y\in \X\\|\supp_{\B}(y)\le n\\ y=P_{\B,\supp_{B}(y)}(y)}}\|u-y\|.\nonumber% \label{approx2sg}
		\end{align}
		We have 
		\begin{align}
			\left\Vert z-\sum_{j\in B}b_j\left(\be_j-x_0^*\left(\be_j\right)x_0\right) \right\Vert =&\left\Vert T\left( u-\sum_{j\in A}b_j\be_j\right)\right\Vert\le \|T\|\K\inf_{\substack{y\in \X\\|\supp_{\B}(y)\le n\\ y=P_{\B,\supp_{B}(y)}(y)}}\|u-y\|\nonumber\\
			\le& 2\K \left\Vert u-\sum_{i\in A}a_i\be_i\right\Vert=2\C \left\Vert z-\sum_{i\in A}a_i\left(\be_i-x_0^*\left(\be_i\right)x_0\right)\right\Vert\nonumber\\
			\le& 2\K(1+\epsilon)\inf_{\substack{y\in \Z\\|\supp_{\B_1}(y)|\le n\\y=P_{\B_1,\supp_{\B_1}(y)}(y)}}\|z-y\|.\label{bigsupport}
		\end{align}
		Since $\be_j^*(x_0)=0$ for all $j\in \N$, it follows that $B$ is also a $t$-greedy set for $z$ with respect to $\B_1$. Given that $z\in \Z$ is arbitrary, it follows from \eqref{smallsupport},  \eqref{bigsupport}, Remark~\ref{remarkFDSPmark} and the fact that $\B_1$ has the FDSP that $\B_1$ is a basis for $\Z$ that meets the conditions of Lemma~\ref{lemmastronger} with $\C=2(1+\epsilon)\K$ - and thus, with $\C=2\K$ as $\epsilon$ is arbitrary. It follows that $\B_1$ is almost greedy, with constant depending only on $t$ and $\K$.  By \cite[Lemma 4.7]{BL2020} and the equivalence between almost greediness and quasi-greediness plus (super)democracy (see \cite[Theorem 3.3]{DKKT}, \cite[Theorem 3.1]{BBG}), we obtain the result for $\B_2$. 
	\end{proof}

	\section{$\n$-partially greedy and $\n$-strong partially greedy bases}\label{sectionpartially}
	\color{black}
	
	The results of \cite{O2015}, Proposition~\ref{uncon} and Lemmas~\ref{lemmasemigreedynsemigreedy} and~\ref{lemmastronger} show that some of the known properties that lie between that of being quasi-greedy and greedy are equivalent to their counterparts for sequences with gaps (note that the semi-greedy property lies between them only for Markushevich bases). This suggests the question of whether one can find, amongst the properties generally studied in connection to the TGA, one or more for which this is not the case, so that one can obtain a new class of bases with the property that the rate of convergence of the TGA and WTGA with gaps is improved with respect to that of $\n$-$t$-quasi-greedy bases. \\
	It turns out that the properties of being partially greedy and strong partially greedy meet the criteria. The first one of these properties was introduced in \cite{DKKT} for Schauder bases, whereas the second one was introduced in \cite{BBL} (and recently considered in the context of quasi-Banach spaces in \cite{B}) in order to facilitate the extension of the notion of partial greediness to the framework of Markushevich bases. The main purposes of this section are to extend the concepts of partially greedy and strong partially greedy bases to our context,  study some of their basic properties and their relations to some well-known properties and their natural extensions, characterize the sequences $\n$ for which there are bases that are $\n$-(strong) partially greedy but not partially greedy, and establish the existence of $\n$-partially greedy bases with some properties of interest. We begin with the central definitions, where we also follow \cite{O2015} in extending the notions from the TGA to the WTGA.

	\begin{definition}\label{definitionnpartially}For $0<t\le 1$, we say that a basis $\B$ in a Banach space $\mathbb X$ is $\n$-$t$-partially greedy if there is a positive constant $\C$ such that 
		\begin{eqnarray}\label{npartiallygreedy}
			\Vert x-\G_n^t(x)\Vert \le \C \|x-S_n(x)\|,\; \forall x\in\X, \forall \G_n^t\in\mathcal{G}_n^t, \forall n\in\n.
		\end{eqnarray}
		The minimum $\C$ for which the above inequality holds is the \emph{$\n$-$t$-partially greedy constant} of $\B$, which we denote by $\C_{p,t}$.  
	\end{definition}

	\begin{definition}\label{definitionnstrongpartially}For $0<t\le 1$, we say that a basis $\B$ in a Banach space $\mathbb X$ is $\n$-$t$-strong partially greedy if there is a positive constant $\C$ such that 
		\begin{eqnarray}\label{nstrongpartiallygreedy}
			\Vert x-\G_n^t(x)\Vert \leq \C \min_{0\le k\le n}\|x-S_k(x)\|,\; \forall x\in\X, \forall \G_n^t\in\mathcal{G}_n^t, \forall n\in\n.
		\end{eqnarray}
		The minimum $\C$ for which the above inequality holds is the \emph{$\n$-$t$-strong partially greedy constant} of $\B$, which we denote by $\C_{sp,t}$.  
	\end{definition}

	For $\n=\N$ and $t=1$, we recover the notions of partially greedy and strong partially greedy bases from \cite{DKKT} and \cite{BBL} respectively, and we use the notation $\C_p$ and $\C_{sp}$.
	
	\begin{remark}\label{remarkequiv}\rm Note that any $\C_{sp,t}$-$\n$-$t$-strong partially greedy  basis is also $\C_{sq,t}$-$\n$-$t$-suppression quasi-greedy and $\C_{p,t}$-$\n$-$t$-partially greedy with $\max\{\C_{sq,t},\C_{p,t}\}\le \C_{sp,t}$, whereas any $\C_{p,t}$-$\n$-$t$-partially greedy $\K$-Schauder basis is $\C_{sp,t}$-$\n$-$t$-strong partially greedy with $\C_{sp,t}\le \C_{p,t}(1+\K)$. Moreover, if $\B$ is bimonotone, then $\C_{sp,t}=\C_{p,t}$. \end{remark}

	In this section, we will also study the relation between $\n$-partial greediness and some extensions of the notions of conservativeness and superconservativeness to our context. 
	\begin{defi}\cite[Definition 3.9]{BB2}\label{definitionsupercon}
		We say that a basis $\mathcal{B}$ is $\n$-superconservative in a Banach space $\mathbb X$ if there exists a positive constant $\C$ such that 
		\begin{eqnarray}\label{cons}
			\Vert \mathbf{1}_{\varepsilon A}\Vert \leq \C\Vert \mathbf{1}_{\varepsilon' B}\Vert,
		\end{eqnarray}
		for all $A,B\subset\N$ with $\vert A\vert\leq \vert B\vert$, $\vert A\vert,\vert B\vert\in\n$, $A<B$, and $\varepsilon\in\Psi_A, \varepsilon'\in\Psi_B$. The smallest constant verifying \eqref{cons} is denoted by $\Delta_{sc}$ and we say that $\mathcal B$ is $\Delta_{sc}$-$\n$-superconservative.
		
		If \eqref{cons} is satisfied for $\varepsilon\equiv\varepsilon'\equiv 1$, we say that $\mathcal B$ is $\Delta_c$-$\n$-conservative, where $\Delta_c$ is the smallest constant for which the inequality holds. 
		
		For $n=\N$, we say that $\B$ is $\Delta_{sc}$-superconservative and $\Delta_c$-conservative.
	\end{defi}
	
	While these extensions of the properties of being conservative and superconservative appear to be the most natural ones and are in line with the extensions of the concepts of democracy and superdemocracy in Definition~\ref{definitionsupercon}, it turns out that two other, perhaps less natural conservative-like properties are more closely connected to $\n$-partially greedy and $\n$-strong partially greedy bases, and thus are more useful tools for studying them.

	\begin{definition}\label{definitionnIIconservative}
		We say that a basis $\mathcal{B}$ in a Banach space $\mathbb X$ is $\n$-order-superconservative if there exists a positive constant $\C$ such that 
		\begin{eqnarray}\label{consII}
			\Vert \mathbf{1}_{\varepsilon A}\Vert \leq \C\Vert \mathbf{1}_{\varepsilon' B}\Vert,
		\end{eqnarray}
		for all $A,B\subset \N$ with $\vert A\vert= \vert B\vert$ for which there is $n\in \n$ such that $A\le n<B$, and all $\varepsilon\in\Psi_A, \varepsilon'\in\Psi_B$. The smallest constant verifying \eqref{consII} is denoted by $\Delta_{osc}$ and we say that $\mathcal B$ is $\Delta_{osc}$-$\n$-order-superconservative. If \eqref{consII} is satisfied for $\varepsilon\equiv\varepsilon'\equiv 1$, we say that $\mathcal B$ is $\Delta_{oc}$-$\n$-order-conservative, where $\Delta_{oc}$ is the smallest constant for which the inequality holds. 
		
		For $\n=\N$, we say that $\B$ is $\Delta_{oc}$-order-conservative or $\Delta_{osc}$-order-superconservative.
		
	\end{definition} 
	\begin{remark}\label{remarkIIsupercon}\rm  For Schauder bases, one can replace the condition $|A|=|B|$ with $|A|\le |B|$ in \eqref{definitionnIIconservative}, obtaining an equivalent definition. It follows easily that a Schauder basis is order-superconservative (resp. order-conservative) if and only if it is superconservative (resp. conservative). 
	\end{remark}
	
	We shall see later (Proposition~\ref{propositionboundedpartially}) that $\n$-strong partially greedy bases are not necessarily $\n$-conservative. On the other hand, the implication holds for $\n$-order-conservativeness, as our next lemma shows.

	\begin{lemma}\label{lemmanpartiallygreedynIIconservative}If $\B$ is $\n$-partially greedy, it is $\n$-order-superconservative, with $\Delta_{osc}\le \C_{p}$. 
	\end{lemma}
	\begin{proof}
		Fix $A$, $B$, $n$, $\varepsilon\in\Psi_A, \varepsilon'\in\Psi_B$ as in Definition~\ref{definitionnIIconservative}. Let 
		$$
		D:=\{1,\dots, n\}\setminus A \qquad \text{ and} \qquad  x:=\bff_{\bfe A}+\bff_{ D} +\bff_{\bfe' B}.
		$$
		Since $B\cup D$ is an $n$-greedy set for $x$, we have
		$$
		\|\bff_{\bfe A}\|=\|x-P_{B\cup D}(x)\|\le \C_{p}\|x-S_n(x)\|= \C_{p}\|\bff_{\bfe' B}\|.
		$$
	\end{proof}

	Next, for any sequence $\n$ with arbitrarily large quotient gaps, we construct a Schauder basis that is $\n$-$t$-strong partially greedy for all $0<t\le 1$, but neither quasi-greedy nor conservative, and thus, not partially greedy. In particular, this extends the result obtained in \cite[Proposition 3.1]{O2015} for the $\n$-quasi-greedy property. In addition, our construction shows that the well-known implication quasi-greedy $\Rightarrow$ unconditional for constant coefficients (see \cite{Wo}) does not extend to $\n$-quasi-greedy bases. We will need the following definition. 
	%\begin{definition}\cite[Definition 2.1]{BB2} We say that a basis $\B$ in a Banach space $\mathbb X$ is $\n$-unconditional for constant coefficients if there is $\C>0$ such that 
	%\begin{eqnarray}
	%	\|\bff_{\bfe A}\|\le \C\|\bff_{\bfe' A}\|\nonumber
	%\end{eqnarray}
	%	for all $A\subset \N$ with $|A|\in\n$ and all $\bfe, \bfe'\in \Psi_A$.
	%\end{definition}

	\begin{proposition}\label{propositionmoved}Let $\n$ be a sequence with arbitrarily large quotient gaps. The following hold: 
		\begin{enumerate}[\rm \color{red}i)\color{black}]
			\item \label{mainresultnpartially} There is a space $\X$ with a monotone Schauder basis $\B$ that is $\n$-$t$-strong partially greedy for all $0<t\le 1$, but neither conservative nor quasi-greedy. 
			
			\item \label{nocharacterizationpartially} There is a space $\X$ with a monotone Schauder basis $\B$ that is $\n$-bidemocratic and $\n$-$t$-quasi greedy for all $0<t\le 1$, but is not $\n$-order-conservative - and thus, not $\n$-partially greedy. 
		\end{enumerate}
	\end{proposition}
	\begin{proof}\ref{mainresultnpartially} Given $\n$ with arbitrarily large quotient gaps, define   
		$$
		\mathcal{S}:=\{S\subset \N: |S|\in \n \text{ and } |S|<S\}, 
		$$
		and let $(n_{k_j})_j$ be a subsequence such that for all $j$,
		\begin{equation}
			n_{k_{j}+1}>3 (j+1) n_{k_j}. \label{fastenough}
		\end{equation}
		Let $\X$ be the completion of $\mathtt{c}_{00}$ with the following norm: 
		\begin{equation}
			\|(a_i)_i\|:=\max\left\lbrace\|(a_i)_i\|_{\infty}, \sup_{S \in \mathcal{S}}\sum\limits_{i\in S}|a_i|, \sup_{j\in\N}\sup_{1\le l\le j n_{k_j}}\left|\sum\limits_{i=1}^l a_{n_{k_j}+i}\right|\right\rbrace.\label{longnorm}
		\end{equation}
		It is easy to check that $\B$ is a monotone Schauder basis for $\X$. We will prove that it has the following properties: 
		\begin{enumerate}[\rm \color{blue}a)\color{black}]
			\item\label{nstpartially} For all $0<t\le 1$, $\B$ is $\n$-$t$-strong partially greedy with constant $\C_{sp,t}\le \max\{t^{-1}, 2\}$.
			\item \label{nsuper} $\B$ is $1$-$\n$-superconservative. 
			\item \label{conservative} $\B$ is not conservative. 
			\item \label{nuccc2} $\B$ is not $\n$-unconditional for constant coefficients. Thus, in particular, it is not quasi-greedy. 
		\end{enumerate}	
		
		To prove \ref{nstpartially}, set $B_0=\emptyset$  and, for each $m\in \N$, define 
		$$
		B_m:=\{1,\dots,m\}.
		$$
		Fix $x\in  \X$, $n\in \n$, $0\le m\le n$, and $A$ a $t$-greedy set for $x$ of order $n$. Clearly, 
		\begin{equation}
			\|(\be_i^*(x-P_A(x))_i\|_{\infty}\le t^{-1}\|(\be_i^*(x-S_m(x))_i\|_{\infty}.\label{easypart2}
		\end{equation}
		Now fix $S\in \mathcal{S}$. If $|S \cap ( B_m\setminus A)|\le |S \cap ( A\setminus B_m)|$, then
		\begin{eqnarray}
			\sum\limits_{i\in S}|\be_i^*(x-P_A(x))|&=&\sum\limits_{i\in S\setminus (A\cup B_m)}|\be_i^*(x-P_A(x))|+\sum\limits_{i\in S\cap (B_m\setminus A)}|\be_i^*(x)|\nonumber\\
			&\le& \sum\limits_{i\in S\setminus (A\cup B_m)}|\be_i^*(x-S_m(x))|+|S \cap ( B_m\setminus A)|\max_{i\in S\cap (B_m\setminus A)}|\be_i^*(x)|\nonumber\\
			&\le &\sum\limits_{i\in S\setminus (A\cup B_m)}|\be_i^*(x-S_m(x))|+ t^{-1}|S \cap ( A\setminus B_m)|\min_{i\in S\cap (A\setminus B_m)}|\be_i^*(x)|\nonumber\\
			&\le&t^{-1}\left(\sum\limits_{i\in S\setminus (A\cup B_m)}|\be_i^*(x-S_m(x))|+\sum\limits_{i\in S\cap (A\setminus B_m)}|\be_i^*(x)|\right) \nonumber\\
			&=& t^{-1}\sum\limits_{i\in S}|\be_i^*(x-S_m(x))|\le t^{-1}\|x-S_m(x)\|.\label{moreinAnotB}
		\end{eqnarray}
		If $|S \cap ( B_m\setminus A)|> |S \cap ( A\setminus B_m)|$, consider the equations 
		$$
		|A|=|A \setminus (B_m\cup S)|+|S \cap ( A\setminus B_m)|+|A\cap B_m\cap S|+|(A\cap B_m)\setminus S|
		$$
		and 
		$$
		|B_m|=|B_m \setminus (A\cup S)|+|S \cap ( B_m\setminus A)|+|A\cap B_m\cap S|+|(A\cap B_m)\setminus S|.
		$$
		Together with the hypothesis that $|A|=n\ge m=|B_m|$, they entail that 
		$$
		|A \setminus (B_m\cup S)|+|S \cap ( A\setminus B_m)|\ge |S \cap ( B_m\setminus A)|. 
		$$
		As $|S \cap ( B_m\setminus A)|> |S \cap ( A\setminus B_m)|$, it follows that there is a nonempty set $D\subset A \setminus (B_m\cup S)$ such that 
		\begin{equation*}
			|S \cap ( A\setminus B_m)|+|D|=|S \cap ( B_m\setminus A)|,%\label{equalize}
		\end{equation*}
		and thus also $E\subset \N$ such that 
		$$
		|E|+|D|=|S  \cap ( B_m\setminus A)| 
		$$
		and either $E=\emptyset$ or $E>A\cup B_m\cup S$. Let
		$$
		S':=(S \setminus (S\cap (B_m\setminus A))\cup D\cup E=(S\setminus (B_m\setminus A))\cup D\cup E. 
		$$
		It is easy to check that $|S'|=|S|\in \n$. The fact that $B_m<D\cup E$ implies that $\min{S}\le \min{S'}$, so $S'\in \mathcal{S}$. Thus,
		\begin{eqnarray}
			\sum\limits_{i\in S}|\be_i^*(x-P_A(x))|&=&\sum\limits_{i\in S\setminus (A\cup B_m)}|\be_i^*(x-P_A(x))|+\sum\limits_{i\in S\cap (B_m\setminus A)}|\be_i^*(x)|\nonumber\\
			&\le& \sum\limits_{i\in S\setminus (A\cup B_m)}|\be_i^*(x-S_m(x)))|+|S \cap ( B_m\setminus A)|\max_{i\in S\cap (B_m\setminus A)}|\be_i^*(x)|\nonumber\\
			&\le &\sum\limits_{i\in S\setminus (A\cup B)}|\be_i^*(x-S_m(x))|+ t^{-1}|D \cup (S \cap ( A\setminus B_m))|\min_{i\in D\cup (S\cap (A\setminus B_m))}|\be_i^*(x)|\nonumber\\
			&\le&t^{-1}\left(\sum\limits_{i\in S\setminus (A\cup B_m)}|\be_i^*(x-S_m(x))|+\sum\limits_{i\in D\cup E \cup (S\cap (A\setminus B_m))}|\be_i^*(x)|\right) \nonumber\\
			&=&t^{-1}\left(\sum\limits_{i\in S'\setminus (A\cup D\cup E)}|\be_i^*(x-S_m(x))|+\sum\limits_{i\in D\cup E \cup (S\cap A)}|\be_i^*(x-S_m(x))|\right) \nonumber\\
			&=& t^{-1}\sum\limits_{i\in S'}|\be_i^*(x-S_m(x))|\le t^{-1}\|x-S_m(x)\|.\label{moreinAnotB2}
		\end{eqnarray}
		From \eqref{moreinAnotB} and \eqref{moreinAnotB2} it follows that 
		\begin{equation}
			\sup_{S\in\mathcal{S}}\sum\limits_{i\in S}|\be_i^*(x-P_A(x))|\le t^{-1}\|x-S_m(x)\|.\label{finallypart1}
		\end{equation}
		Finally, fix $j\in \N$ and $1\le l\le j n_{k_j}$. We consider first the case $|A|\le n_{k_j}$: choose a set $S\subset \N$ so that 
		$$
		|S|=n_{k_j},\qquad n_{k_j}<S, \qquad A\cap \{n_{k_j}+1,\dots, n_{k_j}+l\}\subset S.
		$$
		As $S\in \mathcal{S}$ and $B_m<m+1\le n+1=|A|+1\le n_{k_j}+1$, we have
		\begin{eqnarray}
			\left|\sum\limits_{i=1}^{l}\be_{n_{k_j}+i}^*(x-P_A(x))\right|&\le&\left|\sum\limits_{i=1}^{l}\be_{n_{k_j}+i}^*(x)\right|+\sum\limits_{i=1}^{l}|\be_{n_{k_j}+i}^*(P_A(x))|\nonumber\\
			&\le& \left|\sum\limits_{i=1}^{l}\be_{n_{k_j}+i}^*(x-S_m(x)))\right|+\sum\limits_{i\in S}|\be_{n_{k_j}+i}^*(x-S_m(x))|\nonumber\\
			&\le& 2\|x-S_m(x)\|.\label{largerthanA}
		\end{eqnarray}
		Suppose now that $|A|> n_{k_j}$, and let 
		$$
		D:=\{n_{k_j}+1,\dots, n_{k_j}+l\}.
		$$
		As $|A|\in \n$, $|A|\ge n_{k_j+1}$, so it follows from \eqref{fastenough} that $D\subset \{1,\dots,|A|\}$. If $D\subset A$, then 
		\begin{equation}
			\left|\sum\limits_{i=1}^{l}\be_{n_{k_j}+i}^*(x-P_A(x))\right|=0\le \|x-S_m(x)\|.\label{iszero}
		\end{equation}
		If $D\not \subset A$, there is $S\subset\N$ such that 
		$$
		|S|=|A|<S \qquad \text{ and }\qquad |D\setminus A|= |A\cap S|. 
		$$
		Since $S\in \mathcal{S}$, we have 
		\begin{eqnarray}
			\left|\sum\limits_{i=1}^{l}\be_{n_{k_j}+i}^*(x-P_A(x))\right|&=&\left|\sum\limits_{i\in D\setminus A}\be_{i}^*(x)\right|\le |D\setminus A|\max_{i\in D\setminus A}|\be_i^*(x)|\nonumber\\
			&\le& t^{-1}|A\cap S|\min_{i\in A\cap S}|\be_i^*(x)|\le t^{-1}\sum\limits_{i\in S}|\be_i^*(x)|\nonumber\\
			&= &t^{-1}\sum\limits_{i\in S}|\be_i^*(x-S_m(x))|\le t^{-1}\|x-S_m(x)\|. \label{finalstepforpartially}
		\end{eqnarray}
		From \eqref{easypart2}, \eqref{finallypart1}, \eqref{largerthanA}, \eqref{iszero} and \eqref{finalstepforpartially} it follows that 
		$$
		\|x-P_A(x)\|\le \max\{t^{-1}, 2\}\|x-S_m(x)\|, 
		$$
		and then, the proof of \ref{nstpartially} is complete. \\
		
		Next, we prove \ref{nsuper}: choose $A, B, \bfe\in \Psi_A, \bfe'\in \Psi_B$ as in Definition~\ref{definitionsupercon}, and $S\subset B$ with $|S|=|A|\in \n$. Since $A<B$,  $|S|=|A|<S$. Thus, $S\in \mathcal{S}$. It follows that
		$$
		\|\bff_{\bfe A}\|\le |A|=|S|=\sum\limits_{i\in S}|\be_i^*(\bff_{\bfe' B})|\le \|\bff_{\bfe' B}\|,
		$$
		so $\B$ is $1$-$\n$-superconservative. \\	
		To prove \ref{conservative} define, for each $j$, 
		$$
		D_{j}:=\{n_{k_j}+1,\dots ,(j+1)n_{k_j}\}\qquad \text{and} \qquad E_{j}:=\{(j+1)n_{k_j}+1,\dots ,2(j+1)n_{k_j}\}.
		$$
		We have 
		\begin{equation}
			\|\bff_{D_j}\|\ge \sum\limits_{i=1}^{j n_{k_j}}\be_{n_{k_j+i}}^*(\bff_{D_j})=|D_j|= j n_{k_j}.\label{D_j}
		\end{equation}
		For any $j'\le j$,  
		$$
		(j'+1) n_{k_{j'}}\le (j+1)n_{k_{j}}< E_j,
		$$
		whereas \eqref{fastenough} entails that for any $j'>j$, 
		$$
		E_j<n_{k_{j'}}. 
		$$
		It follows that 
		$$
		\sup_{j'\in\N}\sup_{1\le l\le j' n_{k_{j'}}}\left|\sum\limits_{i=1}^l\be_{n_{k_{j'}}+i}^*(\bff_{E_j})\right|=0. 
		$$
		Similarly, for any $S\in \mathcal{S}$ with $|S|\ge n_{k_j+1}$, by \eqref{fastenough} we have 
		$$
		E_j<S. 
		$$
		Hence, 
		$$
		\|\bff_{E_j}\|=\max\{\|\be_i^*(\bff_{E_j})\|_{\infty}, \sup_{\substack{S \in \mathcal{S}\\ |S|\le n_{k_j}}}\sum\limits_{i\in S}|\be_i^*(\bff_{E_j})|\}\le n_{k_j}.
		$$
		It follows from this and \eqref{D_j} that 
		$$
		\frac{\|\bff_{D_j}\|}{\|\bff_{E_j}\|}\ge j.
		$$
		As $|D_j|=j n_{k_j}=|E_j|$ and $D_j<E_j$ for all $j$, the basis is not conservative.\\
		Finally, let us prove \ref{nuccc2}: fix $j>1$ and let 
		$$\bfe\in \Psi_{B_{n_{{k_j}+1}}}$$
		be a choice of alternating signs. Clearly, 
		\begin{equation}
			\sup_{j'\in\N}\sup_{1\le l\le j' n_{k_{j'}}}\left|\sum\limits_{i=1}^l\be_{n_{k_{j'}}+i}^*(\bff_{\bfe B_{n_{{k_j}+1}}} )\right|=1=\|(\be_i^*(\bff_{\bfe B_{n_{{k_j}+1}}}))_i\|_{\infty}\label{alternating}
		\end{equation}
		As $S\cap  B_{n_{{k_j}+1}}=\emptyset$ for all $S\in \mathcal{S}$ such that $|S|\ge n_{{k_j}+1}$, we have
		\begin{equation}
			\sup_{S \in \mathcal{S}} \sum\limits_{i\in S}|\be_i^*(\bff_{\bfe B_{n_{{k_j}+1}}})|= \sup_{\substack{S \in \mathcal{S}\\ |S|\le  n_{{k_j}}}} \sum\limits_{i\in S}|\be_i^*(\bff_{\bfe B_{n_{{k_j}+1}}})|\le n_{k_j}.\label{uncbutsmall}
		\end{equation}
		On the other hand, 
		$$
		\|\bff_{  B_{n_{{k_j}+1}}}\|\ge \sum\limits_{i=1}^{j n_{k_j}}\be_{n_{k_{j}}+i}^*(\bff_{ B_{n_{{k_j}+1}}} )=jn_{k_j}, 
		$$
		which, when combined with \eqref{alternating} and \eqref{uncbutsmall} gives
		$$
		\frac{\|\bff_{  B_{n_{{k_j}+1}}}\|}{\|\bff_{\bfe  B_{n_{{k_j}+1}}}\|}\ge j 
		$$
		for all $j$. Hence, $\B$ is not $\n$-unconditional for constant coefficients, and the proof is complete. \\
		
		\ref{nocharacterizationpartially} See Example~\ref{examplendemocratic}.
	\end{proof}
	Proposition~\ref{propositionmoved} shows the existence of $\n$-partially greedy bases that are not quasi-greedy for all sequences with arbitrarily large quotient gaps, but the bases we obtain do not have any of the usual properties studied in connection to the TGA, such as conservativeness, democracy or unconditionality for constant coefficients. This suggest the question of whether it is possible to construct $\n$-partially greedy bases that have some of those properties, but are still not quasi-greedy. For example, we can ask whether it is possible to extend \cite[Proposition 3.2]{O2015} and construct $\n$-partially greedy Schauder bases that are superdemocratic and not quasi-greedy. We answer this question in the afirmative and extend the result further, constructing bidemocratic bases with said properties. To that end, we modify the proof of \cite[Proposition 3.17]{AABBL2021}.

	\begin{proposition}\label{theorembidemocraticnotqG2}
		Let $1<p<+\infty$ and $\epsilon>0$. There is a Banach space $\X$ with a bimonotone Schauder basis $\B$ that has the following properties:  
		\begin{enumerate}[\rm \color{red} i) \color{black}]
			\item \label{1bidem22}
			For all finite sets $A\subset \N$ and all $\bfe\in  \Psi_A$, 
			$$
			\left|\left|\bff_{\bfe A}\right|\right|=\left|A\right|^{\frac{1}{p}}\qquad\text{and}\qquad \left|\left|\bff_{\bfe A}^*\right|\right|=\left|A\right|^{\frac{1}{p'}}. 
			$$
			Therefore, $\B$ is $1$-bidemocratic. 
			\item \label{notnQG222} $\B$ is not quasi-greedy.  
			\item \label{tnqg22}There is a sequence $\n=(n_k)_{k\in\N}$ such that, for all $0<t\le 1$, $\B$ is $\C_{sp,t}$-$\n$-strong partially greedy, with 
			$$
			\C_{sp,t}\le \max\left\lbrace 1+\epsilon,\frac{\epsilon}{t} \right\rbrace.
			$$
		\end{enumerate}
	\end{proposition}
	\begin{proof}
		We will assume $\epsilon<\frac{1}{2}$.\\
		Pick any positive integers $2<m_{1,1}<s_{1,1}$ so that 
		\begin{equation}
			\sum_{j=m_{1,1}}^{s_{1,1}}\frac{1}{j}<1.
		\end{equation}
		Suppose that $l_0\ge 1$ and $\left\lbrace m_{l,k}, s_{l,k} \right\rbrace_{\substack{1\le l\le l_0\\ 1\le k\le l }}$ have been chosen so that for all $1\le l\le l_0$: 
		\begin{align}
			\sum_{j=m_{l,k}}^{s_{l,k}}\frac{1}{j}<&\frac{1}{k}\le \frac{1}{l}+\sum_{j=m_{l,k}}^{s_{l,k}}\frac{1}{j}&&\forall 1\le k\le l, \label{conditionclosesums2}\\
			l+1<&m_{l,k}<s_{l,k} &&\forall 1\le k\le l,\label{movetotheright12}\\
			s_{l,k}<&m_{l,k+1} &&\forall 1\le k<l.\label{movetotheright22}
		\end{align}
		Then choose $\left\lbrace m_{l_0+1,k}, s_{l_0+1,k} \right\rbrace_{1\le k\le l_0+1}$ as follows: pick $l_0+2<m_{l_0+1,1}<s_{l_0+1,1}$ so that 
		\begin{equation*}
			\sum_{j=m_{l_0+1,1}}^{s_{l_0+1,1}}\frac{1}{j }<1\le \frac{1}{l_0+1}+\sum_{j=m_{l_0+1,1}}^{s_{l_0+1,1}}\frac{1}{j}.
		\end{equation*}
		If $1\le k_0\le l_0$  ahd  $\left\lbrace m_{l_0+1,k}, s_{l_0+1,k} \right\rbrace_{1\le k\le k_0}$ have been chosen with the properties that 
		\begin{align*}
			\sum_{j=m_{l_0+1,k}}^{s_{l_0+1,k}}\frac{1}{j}&<\frac{1}{k}\le \frac{1}{l_0+1}+\sum_{j=m_{l_0+1,k}}^{s_{l_0+1,k}}\frac{1}{j}&&\forall 1\le k\le k_0, \\
			l_0+2<&m_{l_0+1,k}<s_{l_0+1,k} &&\forall 1\le k\le k_0,\\
			s_{l_0+1,k}<&m_{l_0+1,k+1} &&\forall 1\le k<k_0,
		\end{align*}
		choose $m_{l_0+1,k_0+1}$ and $s_{l_0+1,k_0+1}$ so that
		$$
		s_{l_0+1,k_0+1}>m_{l_0+1,k_0+1}>s_{l_0+1,k_0}
		$$
		and 
		\begin{align*}
			\sum_{j=m_{l_0+1,k_0+1}}^{s_{l_0+1,k_0+1}}\frac{1}{j}<&\frac{1}{k_0+1}\le \frac{1}{l_0+1}+\sum_{j=m_{l_0+1,k_0+1}}^{s_{l_0+1,k_0+1}}\frac{1}{j}.
		\end{align*}
		In this manner, an inductive construction yields sequences of positive integers $\left\lbrace m_{l,k} \right\rbrace_{\substack{l\in \N\\ 1\le k\le l }}$ and $\left\lbrace s_{l,k} \right\rbrace_{\substack{l\in \N\\ 1\le k\le l }}$ such that \eqref{conditionclosesums2}, \eqref{movetotheright12} and \eqref{movetotheright22} hold for all $l\in \N$. \\
		Next, choose a sequence of intervals of positive integers $\left\lbrace A_l\right\rbrace_{l\in\N}$ so that 
		\begin{equation}
			\left|A_l\right|=2l+\sum_{k=1}^{l}s_{l,k}-m_{l,k} \qquad\text{and}\qquad \max\left(A_l\right)+l+1<A_{l+1}\quad\forall l\in \N.  \label{separatingthesets2}
		\end{equation}
		For each $l\in \N$ and each $1\le k\le l$, define 
		
		\begin{align*}
			i_{l,k}:=&\begin{cases}
				\min A_l & \text{if }k=1,\\
				\min A_l+\sum_{r=1}^{k-1}\left(s_{l,r}-m_{l,r}+2\right) & \text{if }1<k\le l,  
			\end{cases}
		\end{align*}
		and for each $l\in \N$ and each $j\in A_l$,
		\begin{align*}
			b_{l,j}:=&
			\begin{cases}
				\frac{1}{k^{\frac{1}{p'}}} & \text{if }j=i_{l,k}\text{ with } 1\le k\le l,\\
				\frac{1}{\left(m_{l,k}+t\right)^{\frac{1}{p'}}} & \text{if } j=i_{l,k}+1+t \text{ with } 1\le k\le l \text{and }0\le t\le s_{l,k}-m_{l,k}.
			\end{cases}
		\end{align*}

		Note that with these definitions, \eqref{movetotheright12} and \eqref{movetotheright22} guarantee that for each $l\in \N$, if $j, j'\in A_l$ and $j\not=j'$, then $b_{l,j}\not=b_{l,j'}$. In particular, for each nonempty finite set $A\subset\N$ and each $l\in \N$, we have
		
		\begin{align}
			\sum\limits_{j\in A\cap A_l}b_{l,j}\le& \sum_{j=1}^{\left|A\right|}\frac{1}{j^{\frac{1}{p'}}} \le \int_{0}^{\left|A\right|}x^{-\frac{1}{p'}}dx = p \left|A\right|^{\frac{1}{p}}.\label{firstboundforfundamentalfunction2}
		\end{align}
		Now choose strictly increasing sequences of positive integers $\mathbf{l}=(l_k)_{k\in\N}$, $\mathbf{f}=\left(f_k\right)_{k\in\N}$ and $\n=(n_k)_{k\in \N}$ as follows: Set $l_1:=1$ and $f_1:=|A_1|$, and choose $n_1>A_{l_1}$. Then pick $l_2>n_1$ sufficiently large so that 
		$$
		\sum_{j=1}^{l_2}\frac{1}{j}>\left(\frac{2^{1+\frac{1}{p}} p n_1}{\epsilon}\right)^{p'}, 
		$$
		and let $f_2:=|A_{l_2}|$, and choose $n_2>A_{l_2}$. Note that $f_2>l_2$. Fix $k_0\ge 2$, and suppose (considering  \eqref{separatingthesets2}) that $(l_k)_{1\le k\le k_0}$, $(f_k)_{1\le k\le k_0}$ and $(n_k)_{1\le k\le k_0}$ have been chosen so that for all $1\le k\le k_0-1$,
		\begin{align}
			&n_{k+1}>\max\left(A_{l_{k+1}}\right)> f_{k+1}=\left\vert A_{l_{k+1}}\right\vert  >l_{k+1}>n_k>\max\left(A_{l_{k}}\right)>f_k=\left\vert A_{l_{k}}\right\vert;\label{intercalate}\\
			&\sum_{j=1}^{l_{k+1}}\frac{1}{j}>\left(\frac{2^{\frac{1}{p}}(k+1)pn_k}{\epsilon}\right)^{p'}.\label{newlowerbound}
		\end{align}
		Then choose $n_{k_0+1}>l_{k_0+1}>n_{k_0}$ so that, setting $f_{k_0+1}=\left\vert A_{l_{k_0+1}}\right\vert$, the  above inequalities hold for $k=k_0$. By induction, this process gives three sequences $\n$, $\mathbf{l}$ and $\mathbf{f}$ such that for all $k\in \N$, \eqref{intercalate} and \eqref{newlowerbound} hold. Let $\|\cdot\|_{\diamond}$ be the seminorm on $\mathtt{c}_{00}$ given by 
		\begin{align}
			\left|\left|\left(a_j\right)_{j\in \N}\right|\right|_{\diamond}=&\frac{\epsilon}{p} \sup_{m\in\N} \frac{1}{n_m} \sup_{\substack{j_1, j_2\in A_{l_{m+1}}\\j_1\le j_2} }\left|\sum_{\substack{j_1\le r\le j_2}}a_{r}b_{l_{m+1},r} \right|,\label{Snorm02}
		\end{align}
		let $\X$ be the completion of $\mathtt{c}_{00}$ with the norm
		\begin{align}
			\left|\left|x\right|\right|=&\max\left\lbrace \left|\left|x\right|\right|_p,\left|\left|x\right|\right|_{\diamond}\right\rbrace,\label{norm2}
		\end{align}
		and let $\B=(\be_i)_i$ be the canonical unit vector basis with dual basis $\B^*=(\be_i^*)_i$; note that $\B$ is a bimonotone normalized Schauder basis for $\X$. \\
		To prove \ref{1bidem22}, fix $A\subset \N$ a nonempty finite set, and $\bfe \in \Psi_A$. By \eqref{firstboundforfundamentalfunction2} we have
		\begin{align*}
			\left|\left|\bff_{\bfe A}\right|\right|_{\diamond}\le \frac{1}{p}\sup_{l\in \N}\sum_{j\in A\cap A_l}\left|\be_j^*\left(\bff_{\bfe A}\right)\right|b_{l,j}\le \left|A\right|^{\frac{1}{p}}=\left|\left|\bff_{\bfe A}\right|\right|_{p}. 
		\end{align*}
		Hence, 
		\begin{align}
			\left|\left|\bff_{\bfe A}\right|\right|=&\left|\left|\bff_{\bfe A}\right|\right|_p=\left|A\right|^{\frac{1}{p}}.\label{fundamentalinX2}
		\end{align}
		Now choose $\bfe' \in \Psi_A$ so that $\varepsilon_j\varepsilon'_j=1$ for all $j\in A$. As \eqref{fundamentalinX2} also holds for if we replace $\bfe$ by $\bfe'$, we have
		\begin{align*}
			\left|A\right|=&\bff_{\bfe A}^*\left(\bff_{\bfe' A}\right)\le \left|\left|\bff_{\bfe A}^*\right|\right|\left|\left|\bff_{\bfe' A}\right|\right|\le \left|\left|\bff_{\bfe A}^*\right|\right|\left|A\right|^{\frac{1}{p}}, 
		\end{align*}
		so 
		\begin{align}
			\left|\left|\bff_{\bfe A}^*\right|\right|\ge& \left|A\right|^{\frac{1}{p'}}. \label{onesidep'2}
		\end{align}
		On the other hand, for all $x\in \X$ with finite support, we have $\left|\left|x\right|\right|\ge \left|\left|x\right|\right|_p$. Since $\B$ is a monotone Schauder basis, this entails that for all $x^*\in \X^*$ with finite support, $\left|\left|x^*\right|\right|\le \left|\left|x^*\right|\right|_{p'}$, so in particular $\left|\left|\bff_{\bfe A}^*\right|\right|\le \left|A\right|^{\frac{1}{p'}}$ which, when combined with \eqref{onesidep'2}, gives $\left|\left|\bff_{\bfe A}^*\right|\right|= \left|A\right|^{\frac{1}{p'}}$. Thus, we have obtained both equalities in \ref{1bidem22}, and it follows from them that $\B$ is $1$-bidemocratic. \\
		To prove \ref{notnQG222}, define, for each $l\in \dl$ and every $j\in A_{l}$,
		\begin{align*}
			a_{l,j}:=&
			\begin{cases}
				\frac{1}{k^{\frac{1}{p}}} & \text{ if }j=i_{l,k} \text{ with } 1\le k\le l, \\
				-\frac{1}{\left(m_{l,k}+t\right)^{\frac{1}{p}}} & \text{ if }j=i_{l,k}+1+t \text{ with } 1\le k\le l \text{ and }0\le t\le s_{l,k}-m_{l,k}.
			\end{cases}
		\end{align*}
		Now for each $l\in \dl$, let 
		$$
		\alpha_l:=\left(\sum_{j\in A_l}|a_{l,j}|^p\right)^{\frac{1}{p}}\qquad \text{and}\qquad x_l:=  \frac{1}{\alpha_l} \sum_{j\in A_l}a_{l,j}\be_j. 
		$$
		Clearly, $\alpha_l\ge 1$ and $\|x_l\|_p=1$ for all $l\in \dl$. Let us see that $\left\lbrace x_l\right\rbrace_{d\in \dl}$ is normalized: For $l\in \dl$, $x_l$ is supported in $A_l$, so the triangle inequality and the facts that $\epsilon<\frac{1}{2}$ and $p\alpha_l\ge 1$ entail that 
		\begin{align}
			\left|\left|x_l\right|\right|_{\diamond}\le &\frac{\epsilon}{p \alpha_l}\sup_{\substack{j_1,j_2\in A_l\\j_1\le j_2 }}\left|\sum_{j_1\le r\le j_2}a_{l,r}b_{l,r} \right|\le \sup_{j\in A_l}\left|\sum_{\substack{ r\in A_{l}\\ r\le j}}a_{l,r}b_{l,r} \right|. \label{diamondonormforxl2}
		\end{align}
		Note that for each $j\in A_l$,
		\begin{align*}
			a_{l,j}b_{l,j}=&
			\begin{cases}
				\frac{1}{k} & \text{ if }j=i_{l,k}\text{ with }1\le k\le l,\\
				-\frac{1}{\left(m_{l,k}+t\right)} & \text{ if }j=i_{l,k}+1+t\text{ with }1\le k\le l \text{ and }0\le t\le s_{l,k}-m_{l,k}.
			\end{cases}
		\end{align*}
		
		From the above equalities and \eqref{conditionclosesums2} it follows that for each $j\in A_l$, the following hold: 
		\begin{itemize}
			\item If  $j=i_{l,1}$, then 
			\begin{align*}
				\left|\sum_{\substack{r\in A_l\\r\le j}}a_{l,r}b_{l,r} \right|=a_{l,i_{l,1}}b_{l,i_{l,1}}=1.
			\end{align*}
			\item If there is $0\le d\le s_{l,1}-m_{l,1}$ such that $j=i_{l,1}+1+d$, then 
			\begin{align*}
				\left|\sum_{\substack{r\in A_l\\r\le j}}a_{l,r}b_{l,r} \right|=1-\sum_{t=0}^{d}\frac{1}{m_{l,1}+t}<1. 
			\end{align*}
			\item If there is $1<k\le l$ such that $j=i_{l,k}$, then 
			\begin{align*}
				\left|\sum_{\substack{r\in A_l\\r\le j}}a_{l,r}b_{l,r} \right|=&\left|\sum_{u=1}^{k-1}\left(\frac{1}{u}-\sum_{t=0}^{s_{l,u}-m_{l,u}}\frac{1}{m_{l,u}+t}\right)+\frac{1}{k}\right|\nonumber\\
				=&\left|\sum_{u=1}^{k-1}\left(\frac{1}{u}-\sum_{t=m_{l,u}}^{s_{l,u}}\frac{1}{t} \right)+\frac{1}{k}\right|\\
				\le&\sum_{u=1}^{k-1}\left|\frac{1}{u}-\sum_{t=m_{l,u}}^{s_{l,u}}\frac{1}{t}\right|+\frac{1}{k}\le \frac{k-1}{l}+\frac{1}{k}\le  \frac{k-1}{k}+\frac{1}{k}=1.
			\end{align*}
			\item If there are $1<k\le l$ and $0\le d\le s_{l,k}-m_{l,k}$ such that $j=i_{l,k}+1+d$, then again by  \eqref{conditionclosesums2}  we have
			\begin{align*}
				&\left|\sum_{\substack{r\in A_l\\r\le j}}a_{l,r}b_{l,r} \right|=\left|\sum_{u=1}^{k-1}\left(\frac{1}{u}-\sum_{t=0}^{s_{l,u}-m_{l,u}}\frac{1}{m_{l,u}+t} \right)\color{white}\right|+\color{white}\left|\color{black}\frac{1}{k} 
				-\sum_{t=0}^{d}\frac{1}{m_{l,k}+d} \right|\\
				&=\sum_{u=1}^{k-1}\left(\frac{1}{u}-\sum_{t=m_{l,u}}^{s_{l,u}}\frac{1}{t} \right)+\frac{1}{k}-\sum_{t=m_{l,k}}^{m_{l,k}+d} \frac{1}{t} \le \frac{k-1}{l}+\frac{1}{k}\le 1. 
			\end{align*}
		\end{itemize}
		Those are all possible cases, so by \eqref{diamondonormforxl2} we conclude that 
		$\left|\left|x_l\right|\right|_{\diamond}\le 1$, and then $\left|\left|x_l\right|\right|=1$.\\
		Now for each $l\in \dl$, define $B_l:=\left\lbrace i_{l,k}\right\rbrace_{1\le k\le l}$. Then $B_l$ is a greedy set for $x_l$. Indeed, given $j\in B_l$, there is $1\le k\le l$ such that $j=i_{l,k}$. Hence, 
		\begin{align*}
			\left|\be_{j}^*\left(x_l\right)\right|=&\frac{1}{\alpha_l}\left|a_{l,i_{l,k}}\right|=\frac{1}{\alpha_l k^{\frac{1}{p}}}\ge \frac{1}{\alpha_l l^{\frac{1}{p}}}.
		\end{align*}
		On the other hand, if $j'\in A_l\setminus B_l$, there is $1\le k\le l$ and $0\le t\le s_{l,k}-m_{l,k}$ such that $j'=i_{l,k}+1+t$. Thus, using \eqref{movetotheright12} we obtain
		\begin{align*}
			\left|\be_{j'}^*\left(x_l\right)\right|=&\frac{1}{\alpha_l\left(m_{l,k}+t\right)^{\frac{1}{p}}} \le \frac{1}{\alpha_l\left(l+1\right)^{\frac{1}{p}}}<\left|\be_{j}^*\left(x_l\right)\right|. 
		\end{align*}
		Let us estimate $\|P_{B_{l_{m+1}}} \left(x_{l_{m+1}}\right)\|$ for each $m\in \N$: By \eqref{conditionclosesums2} we have
		\begin{align}
			\alpha_l^p= \sum_{j\in A_l}|a_{l,j}|^{p}=\sum_{k=1}^{l}\frac{1}{k}+\sum_{k=1}^{l}\sum_{j=m_{l,k}}^{s_{l,k}}\frac{1}{j}\le 2 \sum_{k=1}^{d}\frac{1}{k}\qquad\forall l\in \dl.\nonumber %\label{boundforalphal2}
		\end{align}
		Thus, using \eqref{newlowerbound} we obtain that for each $m\in \N$, 
		\begin{align*}
			\|P_{B_{l_{m+1}}}\left(x_{l_{m+1}}\right)\|\ge&\left|\left|P_{B_{l_{m+1}}}\left(x_{l_{m+1}}\right)\right|\right|_{\diamond}\ge \frac{\epsilon}{p n_m} \left|\sum_{j\in A_{l_{m+1}}}\be_j^*\left(P_{B_{l_{m+1}}}\left(x_{l_{m+1}}\right)\right)b_{l_{m+1},j}\right|\\
			=&\frac{\epsilon}{p n_m}\left(\alpha_{l_{m+1}}\right)^{-1}\sum_{k=1}^{l_{m+1}}a_{l_{m+1}, i_{l_{m+1},k}}b_{l_{m+1}, i_{l_{m+1},k}}\\
			\ge & \frac{\epsilon}{p n_m}\left(2\sum_{k=1}^{d_{m+1}}\frac{1}{k}\right)^{-\frac{1}{p}} \sum_{k=1}^{l_{m+1}}\frac{1}{k}=\frac{2^{-\frac{1}{p}}\epsilon}{p n_m} \left(\sum_{k=1}^{l_{m+1}}\frac{1}{k}\right)^{\frac{1}{p'}}\ge m+1, 
		\end{align*}
		where we used \eqref{newlowerbound} to get the last inequality. Since $\left(P_{B_{l_{m+1}}} \left(x_{l_{m+1}}\right)\right)_{m\in \N}$ is not bounded, $\B$ is not quasi-greedy.\\
		Finally, we prove \ref{tnqg22}: By Remark~\ref{remarkequiv}, it is enough to prove that $\B$  is $\n$-$t$-partially greedy, with constant as in the statement. To that end, fix $x\in \X$, $0<t\le 1$, $k_0\in \N$ and $A$ a $t$-greedy set for $x$ with $|A|=n_{m_{0}}\in \mathbf{n}$, and let $B:=\{1,\dots,n_{m_0}\}$. Clearly we may assume $A\not=B$. Choose $\delta>0$, $m_1\in \N$ and $j_1, j_2\in A_{l_{m_1+1}}$ so that 
		$$
		\left\Vert x-P_A(x)\right\Vert_{\diamond}\le \delta+\frac{\epsilon}{p} \frac{1}{n_{m_1}}\left|\sum_{\substack{r\in A_{l_{m_1+1}}\\j_1\le r\le j_2}}\be_r^*\left(x-P_A (x)\right) b_{l_{m_1+1},r} \right|.
		$$
		If $m_0\le m_1$, then by \eqref{separatingthesets2} and \eqref{intercalate} we get $A_{l_{m_1+1}}>l_{m_1+1}>n_{m_1}\ge n_{m_0}$. Thus, $A_{l_{m_1}+1}\cap B=\emptyset$, which implies that
		\begin{align}
			\left\Vert x-P_A(x)\right\Vert_{\diamond}\le& \delta+\frac{\epsilon}{p} \frac{1}{n_{m_1}}\left\vert \sum_{\substack{r\in A_{l_{m_1+1}}\\j_1\le r\le j_2}}\be_r^*(x-P_A(x)) b_{l_{m_1+1},r} \right\vert\nonumber\\
			=&\delta+\frac{\epsilon}{p} \frac{1}{n_{m_1}}\left\vert \sum_{\substack{r\in A_{l_{m_1+1}}\\j_1\le r\le j_2}}\be_r^*\left(x-P_B(x)-P_A\left(x) \right) \right)b_{l_{m_1+1},r} \right\vert\nonumber\\
			\le&\delta+\frac{\epsilon}{p} \frac{1}{n_{m_1}}\left(\left\vert \sum_{\substack{r\in A_{l_{m_1+1}}\\j_1\le r\le j_2}}\be_r^*\left(x-P_B(x)\right)b_{l_{m_1+1},r} \right\vert+\sum_{r\in A_{l_{m_1+1}}\cap A}\left\vert \be_r^*\left(x\right)\right\vert b_{l_{m_1+1},r} \right)\nonumber\\
			\le& \delta+\left\Vert x-P_B(x)\right\Vert_{\diamond}+\frac{\epsilon}{p} \frac{1}{n_{m_1}}\sum_{r\in A_{l_{m_1+1}} \cap A }\left\vert \be_r^*\left(x-P_B(x)\right)\right\vert b_{l_{m_1+1},r}\nonumber\\
			%\le&  \delta+\left\Vert x-P_B(x)\right\Vert+\frac{\epsilon}{p} \frac{\left\Vert x-P_B(x)\right\Vert}{n_{m_1}}\sum_{r\in A_{l_{m_1+1}} \cap A}b_{l_{m_1+1},r}\nonumber\\
			\le& \delta+\left\Vert x-P_B(x)\right\Vert+\frac{\epsilon}{p} \frac{\left\Vert x-P_B(x)\right\Vert}{n_{m_1}}|A|\le \delta +(1+\epsilon)\left\Vert x-P_B(x)\right\Vert. \label{smallerm0}
		\end{align}
		On the other hand, if $m_0\ge m_1+1$, \eqref{intercalate} implies that  $A_{l_{m_1+1}}\subset\{1,\dots,n_{m_1+1} \}\subset B$. Since 
		\begin{align}
			\left\Vert x-P_B(x)\right\Vert\ge& \left\Vert x-P_B(x)\right\Vert_p\ge  \|P_A(x-P_B(x))\|_{p}\ge \min_{j\in A}\left\vert\be_j^*\left(x\right)\right\vert \left\vert A\setminus B\right\vert^{\frac{1}{p}}\nonumber\\
			\ge& t\max_{j\not\in A}\left\vert\be_j^*\left(x\right)\right\vert \left\vert A\setminus B\right\vert^{\frac{1}{p}}=t\max_{j\not\in A}\left\vert\be_j^*\left(x\right)\right\vert \left\vert B\setminus A\right\vert^{\frac{1}{p}} , \nonumber%\label{AB}
		\end{align}
		it follows that 
		\begin{align}
			\left\Vert x-P_A(x)\right\Vert_{\diamond}\le& \delta+\frac{\epsilon}{p} \frac{1}{n_{m_1}}\left\vert \sum_{\substack{r\in A_{l_{m_1+1}}\\j_1\le r\le j_2}}\be_r^*(x-P_A(x)) b_{l_{m_1+1},r} \right\vert\nonumber\\
			\le &\delta+\frac{\epsilon}{p} \frac{1}{n_{m_1}}\sum_{\substack{r\in A_{l_{m_1+1}}\cap \left(B\setminus A\right)}}\left\vert \be_r^*(x)\right\vert  b_{l_{m_1+1},r}\nonumber\\ 
			\le& \delta+\frac{\epsilon}{p} \frac{1}{n_{m_1}}\frac{\left\Vert x-P_B(x)\right\Vert}{t \left\vert B\setminus A\right\vert^{\frac{1}{p}}}\sum_{r\in A_{l_{m_1+1}}\cap \left(B\setminus A\right)}b_{l_{m+1},r}\nonumber\\
			\le&  \delta+\frac{\epsilon\left\Vert x-P_B(x)\right\Vert}{pt n_{m_1}\left\vert B\setminus A\right\vert^{\frac{1}{p}}}\sum_{k=1}^{\left\vert B\setminus A\right\vert}\frac{1}{k^{\frac{1}{p'}}}\le \delta+\frac{\epsilon\left\Vert x-P_B(x)\right\Vert}{pt n_{m_1}\left\vert B\setminus A\right\vert^{\frac{1}{p}}}\int_{0}^{\left\vert B\setminus A\right\vert}\frac{1}{x^{\frac{1}{p'}}}dx\nonumber\\
			=&\delta+\frac{\epsilon\left\Vert x-P_B(x)\right\Vert}{t n_{m_1}}\le \delta +\frac{\epsilon}{t}\left\Vert x-P_B(x)\right\Vert. \label{biggerm0}
		\end{align}
		Since $\delta$ is arbitrary, combining \eqref{smallerm0} and \eqref{biggerm0} we conclude that $\B$ is $\n$-$t$-partially greedy, with constant as in the statement, and the proof is complete. 
	\end{proof}
	
	We turn our attention now to sequences with bounded quotient gaps. For such sequences, we do have a characterization of $\n$-strong partially greedy  bases that is an extension of that given in \cite[Theorem 3.4]{DKKT}. In order to prove it,  we need two auxiliary lemmas. The first of them is an extension of \cite[Lemma 2.1]{DKO2015} to complex spaces that slightly improves the bound obtained in \cite[Lemma 6.3]{DKO2015}.

	\begin{lemma} \label{lemmaQG}Let $\B$ be a quasi-greedy basis. If $0<t\leq 1$ and $A$ is a $t$-greedy set for $x\in \X$, then 
		$$
		\|P_A(x)\|\le (\C_q+4 t^{-1}\C_q^2)\|x\|.
		$$
	\end{lemma}
	\begin{proof}
		For the case $\F=\R$, this result is just \cite[Lemma 2.1]{DKO2015}, and then the case $\F=\mathbb{C}$ follows by the same argument used in Proposition~\ref{proposition: dem+nqg->ntqg}.
	\end{proof}

	Next, we prove a lemma that holds for any sequence $\n$, and for which we adapt part of the proof of \cite[Theorem 3.4]{DKKT} to our context.

	\begin{lemma}\label{lemmaQGnIIconservative}Let $\B$ be a basis. If $\B$ is quasi-greedy and $\n$-order-conservative, it is $\n$-$t$-partially greedy for all $0<t\le 1$. 
	\end{lemma}
	\begin{proof}
		Fix $x\in  \X$, $n\in \n$, and let $B:=\{i\in \N: i\le n\}$ and $A$ a $t$-greedy set of order $n$ for $x$. If $A=B$,  trivially $\|x-P_A(x)\|\le \|x-S_n(x)\|$. If $A\not=B$, then $B\setminus A\le n<A\setminus B$ and $|B\setminus A|=|A\setminus B|$. Thus,
		
		\begin{eqnarray}
			\|P_{B\setminus A}(x)\|&\stackrel{\text{a)}}{\leq}& 2\kappa\C_q \max_{i\in  B\setminus A}|\be_i^*(x)| \|\bff_{B \setminus A}\|\le 2\kappa t^{-1}\C_q\Delta_{oc}\min_{i\in  A\setminus B}|\be_i^*(x)\|\bff_{A \setminus B}\|\nonumber\\
			&\stackrel{\text{b)}}{\leq}& 8\kappa^2 t^{-1}\C_q^3\Delta_{oc}\|P_{A\setminus B}(x)\|= 8\kappa^2 t^{-1}\C_q^3\Delta_{oc}\|P_{A\setminus B}(x-S_n(x))\|,\nonumber
		\end{eqnarray}
		where in Step a) we used \cite[Proposition 2.1.11]{Bt}, and in Step b) we used \cite[Corollary 2.1.15]{Bt}. Since
		\begin{eqnarray*}
			\|P_{A\setminus  B}(x)\|&=&\|P_{A\setminus B}(x-S_n(x))\|\le (\C_q+4 t^{-1}\C_q^2)\|x-S_n(x)\| \quad\text{ (by Lemma~\ref{lemmaQG})},
		\end{eqnarray*}
		and 
		$$
		\|x-P_A(x)\|\le \|x-S_n(x)\|+\|P_{A\setminus B}(x)\|+\|P_{B\setminus A}(x)\|,
		$$
		the proof is complete. 
	\end{proof}

	Now we can characterize $\n$-strong partially greedy bases for sequences with bounded quotient gaps.

	\begin{proposition}\label{propositionboundedpartially} Let $\B$ be a seminormalized basis. If $\n$ has bounded quotient gaps, the following are equivalent:
		\begin{enumerate}[\rm \color{red}i)\color{black}]
			\item \label{QGC}$\B$ is quasi-greedy and  $\n$-order-superconservative. 
			\item \label{QGCN}$\B$ is $\n$-quasi-greedy and  $\n$-order-conservative.
			\item \label{SnPG}$\B$ is $\n$-strong partially greedy. 
			%\item \label{nPG}$\B$ is $\n$-partially greedy. 
		\end{enumerate}
	\end{proposition}
	\begin{proof}
		The implication \ref{QGC} $\Longrightarrow$ \ref{QGCN} is immediate, whereas \ref{QGCN} $\Longrightarrow$ \ref{QGC} follows by Theorem~\ref{theoremnQGboundedgaps->QG!!!} and the fact that $\n$-order-conservativeness plus unconditionality for constant coefficients entail $\n$-order-superconservativeness.\\
		The implication \ref{SnPG} $\Longrightarrow$ \ref{QGCN} follows by Lemma~\ref{lemmanpartiallygreedynIIconservative}, 
		whereas 	\ref{QGC} $\Longrightarrow$ \ref{SnPG} follows by  Lemma~\ref{lemmaQGnIIconservative}. 
	\end{proof}
	
	It is known that for any sequence $\n$ with bounded quotient gaps and Schauder bases, the properties of being $\n$-democratic, $\n$-conservative, and several others studied in the context of greedy approximation are equivalent to their respective standard counterparts ( \cite[Propositions 2.2, 2.4, 3.6, Lemma 3.10]{BB2}  ). By Theorem~\ref{theoremnQGboundedgaps->QG!!!}, the corresponding equivalence holds for $\n$-quasi-greediness and quasi-greediness, even without a Schauder hypothesis. So, it is natural in this context to ask whether, for such sequences, $\n$-(strong) partially greedy (Schauder) bases are also partially greedy. The results below show that the answers are negative, and characterize the sequences for which equivalence does hold. For our results, we  need the following classification, which was introduced in \cite{BB2} to study extensions of some of the properties that appear naturally in the study of the TGA, such as democracy, unconditionality for constant coefficients and symmetry for largest coefficients.

	\begin{definition}\label{definitiondifferencegaps}
		Let $\n=(n_k)_{k\in \N}$ be a strictly increasing sequence of natural numbers. The additive gaps of $\n$ are the differences $n_{k+1}-n_k$ when $n_{k+1}>n_k+1$. We say that $\n$ has arbitrarily large additive gaps if
		$$\limsup_{k\rightarrow\infty}n_{k+1}-{n_k}=+\infty.$$
		
		\noindent Alternatively, for $l\in\mathbb N_{>1}$, we say that $\n$ has $l$-bounded additive gaps if it has gaps and $ n_{k+1}-n_{k}\le l$ for all $k\in\mathbb N$, and we say that it has bounded additive gaps if it has $l$-bounded additive gaps for some natural number $l\ge 2$.
	\end{definition}
	
	\begin{proposition}\label{propositionshorter} Suppose $\n$ has arbitrarily large additive gaps. There is a Banach space $\X$ with a $1$-unconditional basis $\B$ that is $\n$-partially greedy but not conservative, and thus not partially greedy. 
	\end{proposition}
	\begin{proof}
		Choose a subsequence $(n_{k_j})_j$ so that for all $j$, 
		\begin{equation}
			n_{k_{j}+1}>n_{k_j}+3 \cdot 10^j,  \label{fastenough2}
		\end{equation}
		and a decreasing sequence of positive numbers $(p_k)_k$ so that
		\begin{equation}
			\lim_{k\to +\infty}p_k=1\qquad \text{ and }\qquad \limsup_{j\to +\infty} j\left(\frac{1}{p_{k_j+1}}-\frac{1}{p_{k_j}}\right)=+\infty. \label{slowenough}
		\end{equation}
		For instance, a possible choice is to define 
		$$
		p_{k_{2^j}}:=1+\frac{1}{2j}, \qquad p_{k_{2^j}+1}:=1+\frac{1}{2j+1}
		$$
		for all $j$, and then complete the sequence so that $(p_k)_k$ is decreasing.\\For each $k, j\in \N$, let 
		$$
		\mathcal{S}_k:=\{S\subset \N: |S|=10^{k} \text {and } n_k<S\}, \qquad T_j:=\{n_{k_j}+1,\dots, n_{k_j}+10^j\},
		$$
		and define $\X$ as the completion of $\mathtt{c}_{00}$ with the following norm: 
		\begin{equation}
			\|(a_i)_i\|:=\max\left\lbrace\|(a_i)_i\|_{\infty}, \sup_{k \in \N} \sum\limits_{S\in \mathcal{S}_k}\left(\sum\limits_{i\in S}|a_i|^{p_k}\right)^{\frac{1}{p_k}}, \sup_{j\in\N}\left(\sum\limits_{i\in T_j}|a_i|^{p_{k_j+1}}\right)^{\frac{1}{p_{k_j+1}}}\right\rbrace.\nonumber%\label{longnorm2}
		\end{equation}
		It is clear that the unit vector basis $(\be_i)_i$ is a $1$-unconditional Schauder basis. To prove that it is $\n$-partially greedy, choose $A, B\subset \N$ and $n \in \n$ as in Definition~\ref{definitionnIIconservative}. Fix $j\in \N$, and suppose $T_j\cap A\not=\emptyset$. From \eqref{fastenough2} and the fact that $A\le n$ it follows that $T_j<n_{k_j+1}\le n$. Since $n<B$ and $|A|=|B|$, there is $S\in S_{n_{k_{j}+1}}$ such that $|S\cap B|\ge |T_j\cap A|$. Then, 
		\begin{equation}
			\left(\sum\limits_{i\in T_j}|\be_i^*(\bff_{A})|^{p_{k_j+1}}\right)^{\frac{1}{p_{k_{j}+1}}}\le \left(\sum\limits_{i\in S}|\be_i^*(\bff_{B})|^{p_{k_j+1}}\right)^{\frac{1}{p_{k_{j}+1}}}\le \|\bff_{ B}\|. \label{part12}
		\end{equation}
		Now fix $k\in \N$ and $S\in S_k$. If $A\cap S\not=\emptyset$, then as $A<B$ and $|A|=|B|$ there is $S'\in S_k$ such that $|B\cap S'|\ge |A\cap S|$. Hence, 
		\begin{equation*}
			\left(\sum\limits_{i\in S}|\be_i^*(\bff_{A})|^{p_{k}}\right)^{\frac{1}{p_{k}}}\le \left(\sum\limits_{i\in S'}|\be_i^*(\bff_{B})|^{p_{k}}\right)^{\frac{1}{p_{k}}}\le \|\bff_{ B}\|. 
		\end{equation*}
		From this and \eqref{part12}, taking supremum it follows that 
		$$
		\|\bff_{ A}\|\le \|\bff_{ B}\|.
		$$
		Hence, $\B$ is $1$-$\n$-order-conservative. Since it is also unconditional, by Lemma~\ref{lemmaQGnIIconservative} it is $\n$-partially greedy. To finish the proof, for each $j$ define
		$$
		D_j:=\{n_{k_j}+10^j+1, \dots, n_{k_j}+10^j+10^j \}.
		$$
		For each $j$ and each $k\ge {k_{j}+1}$, by  \eqref{fastenough2} we have $D_j<n_k$, so $D_j\cap S=\emptyset$ for all $S\in S_k$. Thus, 
		$$
		\sup_{k \in \N} \sum\limits_{S\in \mathcal{S}_k}\left(\sum\limits_{i\in S}|\be_i^*(\bff_{D_j})|^{p_k}\right)^{\frac{1}{p_k}}=\sup_{1\le k\le k_j}\sum\limits_{S\in \mathcal{S}_k}\left(\sum\limits_{i\in S}|\be_i^*(\bff_{D_j})|^{p_k}\right)^{\frac{1}{p_k}}\le  |D_j|^{\frac{1}{p_{k_j}}}=  10^{\frac{j}{p_{k_j}}}.
		$$
		Also, by \eqref{fastenough2}, we have $T_i\cap D_j=\emptyset$ for all $i,j$. Hence, 
		$$
		\|\bff_{D_j}\|\le  10^{\frac{j}{p_{k_j}}}
		$$
		for all $j$. On the other hand, 
		$$
		\|\bff_{T_j}\|\ge \left(\sum\limits_{i\in T_j}|\be_i^*(\bff_{T_j})|^{p_{k_j+1}}\right)^{\frac{1}{p_{k_j+1}}}=10^{\frac{j}{p_{k_{j+1}}}}.
		$$
		Therefore, by \eqref{slowenough},
		$$
		\limsup_{j\to +\infty}\frac{\|\bff_{T_j}\|}{\|\bff_{D_j}\|}\ge \limsup_{j\to +\infty} 10^{\frac{j}{p_{k_j+1}}-\frac{j}{p_{k_j}}}=+\infty. 
		$$
		As $T_j<D_j$ and $|T_j|=|D_j|$ for all $j$, we conclude that $\B$ is not conservative, and thus not partially greedy. 
	\end{proof}
	\begin{remark}\label{remarknonconservative}\rm Proposition~\ref{propositionshorter} and \cite[Lemma 3.10]{BB2} together imply that for sequences with bounded quotient gaps but arbitrarily large additive gaps, $\n$-strong partially greedy Schauder bases can fail to be $\n$-conservative, which suggests $\n$-order-conservativeness is in a sense a better extension of conservativeness to our context than the more seemingly natural $\n$-conservativeness. 
	\end{remark}
	For sequences with bounded additive gaps, we have the following result. 
	\begin{lemma}\label{lemmaboundedadditive} Let $\B$ be an $\n$-strong partially greedy Markushevich basis. If $\n$ has bounded additive gaps, it is strong partially greedy. 
	\end{lemma}
	\begin{proof}
		By Lemma~\ref{lemmanpartiallygreedynIIconservative}, $\B$ is $\n$-order-conservative. As $\B$ is $\n$-quasi-greedy, by Theorem~\ref{theoremnQGboundedgaps->QG!!!} (or \cite[Proposition 4.1]{O2015}) it is quasi-greedy. To see that it is conservative, fix $A,B\subset \N$ with $A<B$ and $0<|A|\le |B|$. If $B>n_1$, let 
		$$
		k_1:=\max_{k \in \N}\{n_k<B\}\qquad \text{and} \qquad A_0:=\{i\in  A:i\le n_{k_1}\}. 
		$$
		Note that $|A\setminus A_0|\le m$ because either the set is empty or $$n_{k_1}<A\setminus A_0<\min B\le n_{{k_1}+1}\le n_{k_1}+m.$$ As $A_0\le n_{k_1}<B$, we have
		$$
		\|\bff_{A}\|\le \|\bff_{A_0}\|+\|\bff_{A\setminus A_0}\|\le \Delta_{oc}\|\bff_{B}\|+ \alpha_1\alpha_2m \|\bff_{B}\|.
		$$
		On the other hand, if $B\not >n_1$, then $A<n_1$. Hence, 
		$$
		\|\bff_{A}\|\le \alpha_1\alpha_2(n_1-1)\|\bff_{B}\|. 
		$$
		Thus, $\B$ is conservative. By \cite[Proposition 5.1]{BBL}, it is strong partially greedy. 
	\end{proof}

	\begin{corollary}\label{corollaryequivforboundedadditive} Let $\n$ be a sequence with gaps. The following are equivalent: 
		\begin{itemize}
			\item $\n$ has bounded additive gaps. 
			\item Every $\n$-strong partially greedy Markushevich basis is strong partially greedy. 
		\end{itemize}
	\end{corollary}

	\section{Examples}\label{sectionexamples}
	\color{black}
	In this section, we consider a family of examples from \cite[Proposition 3.1]{O2015}, which we use throughout the paper. We will use some of the results proved for a modified version of this example was studied in \cite{BB2}, as the proofs hold also in this case with only straightforward modifications. 
	\begin{example}\label{examplendemocratic}
		Suppose $\n$ has arbitrarily large quotient gaps, and write $\n=(n_k)_{k=1}^\infty$ and find $k_1<k_2<...$ such that the sequence $(n_{k_i+1}/n_{k_i})_{i=1}^\infty$ increases without a bound. For $i\in\mathbb N$, write
		$$c_i=\left(\frac{n_{k_i+1}}{n_{k_i}}\right)^{1/4},\;\; m_i=\floor{\sqrt{n_{k_i+1}n_{k_i}}}.$$
		Let $\tilde{m}_i = \sum_{j<i}m_i$ (so that $\tilde{m}_1=0$ and $\tilde{m}_{i+1}=\tilde{m}_i+m_i$ for $i\geq 1$), and let $\X$ be completion of $\mathtt{c}_{00}$ with the norm:  
		
		$$\left\Vert\sum_{j}a_j\be_j\right\Vert = \max\left\lbrace \Vert (a_j)_j\Vert_2, \sup_{i\in\mathbb N}\frac{c_i}{\sqrt{m_i}}\max_{i\leq l\leq m_i}\left\vert \sum_{j=\tilde{m}_i+1}^{\tilde{m}_i+l}a_j\right\vert\right\rbrace.$$
		The canonical unit vector basis $\mathcal B=(\be_i)_{i\in\mathbb N}$ is a monotone Schauder basis with the following properties. 
		\begin{enumerate}[\rm \color{red} a)\color{black}]
			\item \label{aex1}$\B$ is $\n$-quasi-greedy and not quasi-greedy.
			\item  \label{cex1} $\B$ is $\Delta_b$-$\n$-bidemocratic with $\Delta_b\le \sqrt{2}$. 
			
			\item  \label{gex1} $\B$ is not $\n$-order-conservative and hence not $\n$-partially greedy.
			
		\end{enumerate}
	\end{example}

	\begin{proof}
		It is clear from the definition that $\B$ is a monotone Schauder basis. \\
		
		\textbf{Step \ref{aex1} $\n$-quasi-greediness:} See \cite[Proposition 3.1]{O2015}.\\

		\noindent\textbf{Step \ref{cex1} $\n$-bidemocracy:} 
		
		Fix $A, B\subset \N$ with $|A|=|B|=n\in \mathbf{n}$,  $\bfe\in \Psi_A, \bfe' \in \Psi_B$. The argument of the proof \cite[Example 5.2]{BB2} gives 
		\begin{equation*}
			\|\mathbf{1}_{\bfe' B}\|\le 2\sqrt{|B|}.
		\end{equation*}
		On the other hand, since $\|x\|\ge\|x\|_2$ for all $x\in \X$, we have $\|x^*\|\le \|x^*\|_2$ for all $x^*\in \X^*$ with finite support. Hence,
		$$
		\|\bff_{\bfe A}^*\|\le \sqrt{|A|}
		$$
		It follows that 
		$$
		\|\bff_{\bfe A}^*\|\|\bff_{\bfe' B}\|\le \sqrt{2} \sqrt{|A|} \sqrt{|B|}=\sqrt{2}n,
		$$
		so $\B$ is $\Delta_b$-$\n$-bidemocratic with $\Delta_b\le \sqrt{2}$. \\

		\noindent \textbf{Step \ref{gex1} $\n$-partial greediness:}

		For each $i\in \N$, let 
		
		\begin{equation*}
			B_i:=\{\widetilde{m}_i+1,\dots, \widetilde{m}_i+m_i\}. 
		\end{equation*}
		We have 
		\begin{equation}
			\|\mathbf{1}_{B_i}\|\ge \left\vert \frac{c_i}{\sqrt{m_i}}\sum\limits_{j=\widetilde{m}_i+1}^{\widetilde{m}_i+m_i}\be_j^*(\mathbf{1}_{B_i})\right\vert= \sqrt{m_i}c_i=c_i\sqrt{|B_i|} \label{Bi}
		\end{equation}
		Note that 
		$$
		\frac{c_i}{\sqrt{m_i}}\le 1 \qquad\forall i\ge 2. 
		$$
		Thus, $\left(\be_{\widetilde{m}_i+1}\right)_{i\ge 2}$ is isometrically equivalent to the unit vector basis of $\ell_2$. Therefore, given $i\in \N$ there is $n\in \n$ and  $D_i>n>B_i$ such that $|D_i|=|B_i|$ and 
		$$
		\|\mathbf{1}_{D_i}\|=\sqrt{|D_i|}.
		$$
		Since $(c_i)_{i\in \N}$ is unbounded, from this and \eqref{Bi} it follows that $\B$ is not $\n$-order-conservative.

	\end{proof}
	
	\begin{remark}\label{remarkunconditional}\rm If we replace the norm in Example~\ref{examplendemocratic} by the norm 
		$$\left\Vert\sum_{j}a_j\be_j\right\Vert = \max\left\lbrace \Vert (a_j)_j\Vert_2, \sup_{i\in\mathbb N}\frac{c_i}{\sqrt{m_i}} \sum_{j=\tilde{m}_i+1}^{\tilde{m}_i+m_i}\left|a_j\right|\right\rbrace,$$
		the resulting basis is $1$-unconditional, and the proofs of all the steps \ref{cex1} and \ref{gex1} hold with only minor, strightforward modifications. In particular, this shows that even for unconditional bases, $\n$-bidemocracy does not entail $\n$-order-conservativeness. 
	\end{remark}

	\section{Open questions}\label{sectionquestions}
	
	\begin{question}\rm For $\n$ with bounded quotient gaps, we characterize $\n$-strong partially greedy  bases as those that are $\n$-quasi-greedy and $\n$-order-conservative (Proposition~\ref{propositionboundedpartially}). Does this characterization hold for $\n$ with arbitrarily large quotient gaps? If not, is there a characterization of $\n$-strong-partially greedy  bases as those $\n$-quasi-greedy that have another, relatively simple property?
	\end{question}

	\begin{question}
		\rm In Section~\ref{sectionsemi}, we proved that $\n$-semi-greedy Markushevich bases are semi-greedy. Does this implication still hold if we remove the totality  hypothesis? 
	\end{question}
	\newpage

	\section*{Annex: Summary of some important constants}
	\begin{table}[ht]
		\begin{center}
			\begin{tabular}{c c c}\hline
				& & \\
				{\bf Symbol}& {\bf Name of constant} &  {\bf Ref. equation}\\
				& & \\
				\hline
				& & \\
				$\C_{q,t}$ & Quasi-greedy constant & \eqref{q} \\ 
				& & \\
				$\C_{sq,t}$ & Suppression-quasi-greedy constant & \eqref{sq} \\ 
				& & \\
				$\C_{p,t}$ & Partially greedy constant &   \eqref{npartiallygreedy} \\ 
				& & \\
				$\C_{sp,t}$ & Strong partially greedy constant &   \eqref{nstrongpartiallygreedy} \\
				& & \\
				$\mathbf{K}_{s}$ & Suppression unconditionality constant &   \eqref{unc} \\ 
				& & \\
				$\Delta$ & Symmetry for largest coeff. constant &  \eqref{sy} \\ 
				& & \\
				$\Delta_b$ & Bidemocracy constant &   \eqref{bi} \\ 
				& & \\
				$\Delta_{osc}$ & Order-superconservativity constant &   \eqref{consII} \\
				& & \\
				$\Delta_{oc}$ & Order-conservativity constant &   \eqref{consII} \\
			\end{tabular}
		\end{center}
	\end{table}

\end{document}